\documentclass[11pt,a4paper]{amsart}

\usepackage{latexsym}
\usepackage{amsmath}
\usepackage{amsthm}
\usepackage{amscd}
\usepackage{amssymb}
\usepackage{hyperref}
\usepackage{mdframed}
\usepackage{graphicx}
\usepackage{hyperref}
\usepackage{enumerate}
\usepackage{mathrsfs}
\usepackage{url}
\usepackage{tikz}
\usepackage{tikz-cd}
\usepackage{calc}

\newtheorem{introtheorem}{Theorem}
\newtheorem{theorem}{Theorem}[section]
\newtheorem{proposition}[theorem]{Proposition}
\newtheorem{lemma}[theorem]{Lemma}
\newtheorem{corollary}[theorem]{Corollary}

\theoremstyle{definition}
\newtheorem{defn}[theorem]{Definition}
\theoremstyle{definition}
\newtheorem{rem}[theorem]{Remark}
\theoremstyle{definition}
\newtheorem{ex}[theorem]{Example}
\theoremstyle{definition}

\numberwithin{equation}{section}

\let\hom\relax\DeclareMathOperator{\hom}{Hom}
\DeclareMathOperator{\spec}{Spec}

\DeclareMathOperator{\NP}{NP}
\DeclareMathOperator{\an}{an}
\DeclareMathOperator{\val}{val}

\DeclareMathOperator{\Haar}{Haar}

\DeclareMathOperator{\can}{can}
\DeclareMathOperator{\tor}{tor}

\DeclareMathOperator{\id}{id}
\DeclareMathOperator{\trop}{trop}

\DeclareMathOperator{\divisor}{div}
\DeclareMathOperator{\primitive}{primitive}
\DeclareMathOperator{\vol}{vol}
\DeclareMathOperator{\MM}{M\mathcal{M}}
\DeclareMathOperator{\MV}{MV}
\DeclareMathOperator{\MI}{MI}
\DeclareMathOperator{\dom}{dom}
\DeclareMathOperator{\stab}{stab}
\DeclareMathOperator{\rec}{rec}

\DeclareMathOperator{\FS}{FS}
\DeclareMathOperator{\Gal}{Gal}
\DeclareMathOperator{\sep}{sep}
\usepackage{amssymb}
\usepackage[all]{xy}

\begin{document}

\title{Heights of hypersurfaces in toric varieties}
\author[Gualdi]{Roberto~Gualdi}
\address{Institut de Math\'ematiques de Bordeaux, Universit\'e de Bordeaux, cours de la Lib\'eration 351, 33405 Talence, France.}
\email{roberto.gualdi@math.u-bordeaux.fr}
\urladdr{\url{https://www.math.u-bordeaux.fr/~robgualdi/}}
\date{}

\maketitle

\begin{abstract}
For a cycle of codimension $1$ in a toric variety, its degree with respect to a nef toric divisor can be understood in terms of the mixed volume of the polytopes associated to the divisor and to the cycle. We prove here that an analogous combinatorial formula holds in the arithmetic setting: the global height of a $1$-codimensional cycle with respect to a toric divisor equipped with a semipositive toric metric can be expressed in terms of mixed integrals of the $v$-adic roof functions associated to the metric and the Legendre-Fenchel dual of the $v$-adic Ronkin function of the Laurent polynomial of the cycle.
\end{abstract}

\setcounter{tocdepth}{1}
\tableofcontents
\section*{Introduction}

%
% setting
%

The arithmetic intersection theory of toric varieties with respect to toric line bundles equipped with their canonical metric was first studied by Maillot in \cite{Mail}. Later, the systematic extension of the toric dictionary to Arakelov geometry was carried out by Burgos Gil, Philippon and Sombra in \cite{BPS}. It turns out from their study that suitably metrized toric line bundles can be expressed in terms of families of concave functions on convex polytopes and that the height of the toric variety with respect to this choice is related to the integral of such functions. Their theory allows to treat a large spectrum of height functions, namely the ones arising from toric line bundles equipped with toric metrics; this includes the canonical heights studied by Maillot and the Fubini-Study height. On the other hand, the techniques developed in \cite{BPS} only apply to the computation of the height of toric subvarieties and do not solve, for instance, the problem of determining the height of a general cycle of codimension $1$. This question was answered in a very special case by \cite{Mail}, where a relation between the canonical height of a hypersurface in a smooth projective toric variety and the Mahler measure of the corresponding polynomial was given. Other computations have been performed by Cassaigne and Maillot in \cite{CM} for the Fubini-Study height of hypersurfaces in projective spaces. Extending the techniques of \cite{BPS}, we give here a combinatorial formula for the height of a $1$-codimensional cycle in a toric variety for a much more general choice of metrics. 
%
% description of the results
%
\vspace{\baselineskip}
\\For the sake of simplicity, we restrict for the moment to the case of an ambient proper toric variety $X_\Sigma$ of dimension $n$ over $\mathbb{Q}$, leaving the treatment of the case of an arbitrary base adelic field to the body of the paper. Let $\mathfrak{M}$ stand for the set of places of $\mathbb{Q}$. As usual in toric geometry, we denote by $M$ the lattice of characters of the torus of $X_\Sigma$, by $N$ its dual lattice, and by $M_\mathbb{R}$ and $N_\mathbb{R}$ the corresponding real vector spaces. We are interested in a combinatorial expression for the height of a cycle of codimension $1$ in $X_\Sigma$ with respect to a suitable choice of a metrized (Cartier) divisor. By the linearity of height functions, we can restrict to the case of an irreducible hypersurface $Y$. Moreover, since irreducible hypersurfaces in $X_\Sigma$ not intersecting its dense open torus have to coincide with $1$-codimensional toric orbits, whose height has already been calculated in \cite[Proposition 5.1.11]{BPS}, we can assume that the generic point of $Y$ lies in the dense open orbit of $X_\Sigma$. Under this assumption, $Y$ is described by an irreducible Laurent polynomial $f$ with rational coefficients. Its Newton polytope $\NP(f)$ is a nonempty subset of $M_\mathbb{R}$ capturing enough information for the intersection theoretical properties of $Y$. For instance, \hyperref[degree of an hypersurface]{Proposition \ref*{degree of an hypersurface}} implies that the degree of $Y$ with respect to a toric divisor $D$ on $X_\Sigma$ generated by its global sections is given by \[\deg_D(Y)=\MV_M(\Delta,\dots,\Delta,\NP(f)),\]where $\Delta$ is the polytope in $M_\mathbb{R}$ associated to $D$ and $\MV_M$ denotes the mixed volume of convex bodies in $M_\mathbb{R}$ with respect to a suitably normalized Haar measure.
\\The height of a cycle in $X$ is the arithmetic counterpart of its degree with respect to a divisor $D$. Its definition requires as an extra datum the choice of an adelic semipositive metric on $D$, see \hyperref[section about Arakelov on toric varieties]{section \ref*{section about Arakelov on toric varieties}} for a precise definition. To have a combinatorial description of heights in toric varieties, it is necessary to ask $D$ to be a toric divisor (with associated polytope $\Delta$) and the metric on it to be ``toric invariant'' in some sense. In such a situation, Burgos Gil, Philippon and Sombra have shown that a combinatorial description is possible, translating the additional information of the metric into an extra dimension on the convex geometrical side: an adelic semipositive toric metric on $D$ is associated to a family $(\vartheta_v)_{v\in\mathfrak{M}}$ of continuous concave functions on $\Delta$, called the \emph{roof functions} of the metric, such that $\vartheta_v=0$ for all but finitely many $v$. We show how the height of $Y$ with respect to the adelic semipositive toric metrized divisor $\overline{D}$ can be expressed using such an extra dimensional representation, in a spirit analogous to the formula for its degree mentioned above. The key idea consists in associating to the polynomial $f$ defining $Y$, for every place $v$ of $\mathbb{Q}$, a suitable function which we call the \emph{$v$-adic Ronkin function} of $f$ and denote by $\rho_{f,v}$. It is a concave function on $N_\mathbb{R}$ whose value at $u$ can be interpreted as an average of $-\log|f|$ on the fiber of the tropicalization map over $u$. When $v$ is archimedean, it is the Ronkin function studied by Passare and Rullg{\aa}rd among others, while for non-archimedean places it coincides with the $v$-adic tropicalization of the polynomial $f$. Its Legendre-Fenchel dual $\rho_{f,v}^\vee$ is a concave function on $M_\mathbb{R}$ which is supported on the Newton polytope of $f$. Recall now that Philippon and Sombra have introduced in \cite{PS1} a polarized version of the integration of a concave function with bounded support, called the \emph{mixed integral}. For the choice of a suitably normalized Haar measure on the vector space $M_\mathbb{R}$, it is a multilinear symmetric real valued function $\MI_M$ taking as entries $n+1$ concave functions supported on convex bodies in $M_\mathbb{R}$.

\begin{introtheorem}\label{main theorem introduction}
The height of $Y$ with respect to $\overline{D}$ is given by \[h_{\overline{D}}(Y)=\sum_{v\in\mathfrak{M}}\MI_M\big(\vartheta_v,\dots,\vartheta_v,\rho_{f,v}^\vee\big).\]
\end{introtheorem}

Despite the complexity of the computation of the archimedean Ronkin function, the formula in the previous theorem clarifies the relation between the defining polynomial of an irreducible hypersurface and its height with respect to an adelic semipositive toric metrized divisor. 
It is easy to specialize it to the case of the canonical metric on $D$, where it reduces to the equality proved in \cite{Mail}, or of the Fubini-Study metric in the projective setting. We hope that a better understanding of the properties of mixed integrals and archimedean Ronkin functions could be used to deduce both lower and upper bounds for the height of $Y$. More importantly, our result asserts that the collection of the $v$-adic Ronkin functions of a hypersurface contains enough information to determine its height; we wonder whether other arithmetical properties of $Y$ might be read in terms of such functions.
%
% remarks
%
\vspace{\baselineskip}
\\To show the stated result, we prove more precise formulas for the local height and the toric local height of a $1$-codimensional cycle. We also show some new properties of mixed integrals and we propose a more uniform definition and study of $v$-adic Ronkin functions which is independent on whether the place $v$ is archimedean or not. The obtained formulas for the height extend to the case of admissible adelic toric metrized divisors as alternated sums of mixed integrals, as in \cite[Remark 5.1.10]{BPS}.
\\For an arbitrary adelic base field $K$, we remark that one needs to prove that the global height of $Y$ with respect to an adelic semipositive toric metrized divisor is a finite sum and hence well-defined. This is automatic if $K$ is a global field, because of \cite[Proposition 1.5.14 and Theorem 4.9.3]{BPS}. We show it here for an arbitrary adelic field $K$ with product formula, in which case the formulas for the height stay true. In the more general setting of an adelic field $K$ not satisfying the product formula, it is easy to verify that the same equality for the global height holds up to the sum by the \textit{defect} of $K$, see \cite[Definition 1.5.9]{BPS}. Finally, the recent work \cite{GH} suggests that similar statements might hold for a base $M$-field.
%
% resume of the sections
%
\vspace{\baselineskip}
\\We now briefly summarize the content of each section. 
\\In \hyperref[section about convex geometry]{section \ref*{section about convex geometry}}, we recall the tools from convex geometry which are needed throughout all the paper: Legendre-Fenchel duality of concave functions, real Monge-Amp\`ere measures and mixed integrals. In particular, we re-interpret the recursive formula for mixed integrals proved by Philippon and Sombra in terms of mixed real Monge-Amp\`ere measures. We then make use of it to deduce two elementary, though useful, properties of such operators. Finally, we describe what happens when one of the functions appearing in the mixed integral is the indicator function of a line segment.
\\\hyperref[section about Ronkin functions]{Section \ref*{section about Ronkin functions}} deals with the key object of our work: $v$-adic Ronkin functions of Laurent polynomials, which are introduced and described after recalling the needed preliminaries in tropical and non-archimedean geometry. In this context, the discussion of a notion of minimal boundaries allows to treat the archimedean and non-archimedean cases homogeneously. 
\\In \hyperref[section about Arakelov on toric varieties]{section \ref*{section about Arakelov on toric varieties}} we briefly recall the general adelic Arakelov framework and focus then on the results obtained by Burgos Gil, Philippon and Sombra in the toric setting. To keep the treatment of archimedean and non-archimedean places on equal footing, we rephrase their description of the Chambert-Loir measure of semipositive toric metrized divisors in terms of minimal boundaries of tropical fibers.
\\As a needed step for the main proof, we combinatorially describe the Weil divisor of the rational function defined by a Laurent polynomial on a toric variety. This result can be of independent interest and has then been set aside in \hyperref[section divisor of rational functions]{section \ref*{section divisor of rational functions}}.
\\\hyperref[section hypersurface]{Section \ref*{section hypersurface}} is dedicated to the proofs of our main results \hyperref[toric local height of hypersurfaces]{Theorem \ref*{toric local height of hypersurfaces}} and \hyperref[hypersurfaces are integrable and their global height]{Theorem \ref*{hypersurfaces are integrable and their global height}}, which are formulas for the local height and the toric local height of cycles of codimension $1$ in toric varieties. We then make use of them to prove the integrability statement and a formula for their global heights, with respect to the choice of adelic semipositive toric metrized divisors.
\\For binomial hypersurfaces, such a formula is compatible with the one deduced from \cite{BPS}. This is shown in \hyperref[section about examples]{section \ref*{section about examples}}, where we also apply our results to some other particular cases. We provide convex geometrical formulas for the canonical height of $1$-codimensional cycles, obtaining the quoted result by Maillot, and for the Fubini-Study height of a projective hypersurface. We also propose a new height function, the $\rho$-height, for which we give a compact formula. We do not know any application of such a height, which could be anyway worth studying.
%
% terminology and notations
%
\vspace{\baselineskip}
\\\textbf{Terminology and notations.} A \emph{variety} $X$ is assumed to be a reduced and irreducible separated scheme of finite type over a field. By an \emph{irreducible hypersurface} in it we mean a closed integral subscheme of codimension $1$ in $X$. A \emph{divisor} on $X$ is a Cartier divisor, unless otherwise stated. Toric varieties are assumed to be normal; whenever the choice of the base field $K$ is clear from the context, the notation $X_\Sigma$ will refer to the toric variety over $K$ associated to the fan $\Sigma$.
\\The term \emph{measure} on a topological space stands for a signed Borel measure on it; in particular, measures admit a well-defined push-forward via continuous mappings. A measure which only takes non-negative real values on Borel subsets is called a \emph{positive measure}.
%
% acknowledgments
%
\vspace{\baselineskip}
\\\textbf{Acknowledgments.} The author wishes to thank his PhD supervisors Mart\'{\i}n Sombra and Alain Yger for their guidance and many fruitful discussions, as well as C\'esar Mart\'{\i}nez for the interest shown in a preliminary version of this text. Also, he is grateful to the two anonymous referees for the careful reading and the valuable remarks, which highly contributed to improve the quality of the text. The work has been prepared at the Universit\'e de Bordeaux and at the Universitat de Barcelona as part of the author's PhD project, and partially supported by the MINECO research project MTM2015-65361-P and the CNRS project PICS 6381 ``G\'eom\'etrie diophantienne et calcul formel''.

\section{Preliminaries in convex geometry}\label{section about convex geometry}

This section is devoted to recalling notions from convex geometry that will be useful in the sequel. We follow the conventions and notations of \cite[Chapter 2]{BPS}, referring to \cite[\S12]{R} for a more complete treatment of the subject. We refer to these two sources for the proofs of the statements we make here.
\\For the whole section, let $N$ be a lattice of rank $n$ and $M:=\hom(N,\mathbb{Z})$ its dual lattice. Denote by $N_\mathbb{R}=N\otimes_\mathbb{Z}\mathbb{R}$ and by $M_\mathbb{R}=M\otimes_\mathbb{Z}\mathbb{R}$ the corresponding $n$-dimensional real vector spaces.
\\By a \emph{polyhedron} in $N_\mathbb{R}$ we mean a convex subset of $N_\mathbb{R}$ obtained as the intersection of finitely many closed halfspaces $\{u\in N_\mathbb{R}:\langle x,u\rangle+c\geq0\}$, with $x\in M_\mathbb{R}$ and $c\in\mathbb{R}$. If all the slopes $x$ can be chosen in $M$, the polyhedron is said to be \emph{rational}. A \emph{polytope} is a bounded polyhedron. A polytope in $N_\mathbb{R}$ whose vertices all lie in $N$ is called a \emph{lattice polytope}; it is in particular a rational polytope. A compact convex subset of $N_\mathbb{R}$ is called a \emph{convex body}.

\subsection{Legendre-Fenchel duality}

A function $f:N_\mathbb{R}\to\mathbb{R}\cup\{-\infty\}$ is said to be \emph{concave} if it is not identically $-\infty$ and for every $u_1,u_2\in N_\mathbb{R}$ and for every $t\in\left[0,1\right]$, one has the inequality \[f(tu_1+(1-t)u_2)\geq tf(u_1)+(1-t)f(u_2).\] The \emph{effective domain} of a concave function $f$ is the set on which the function takes values different from $-\infty$ and it is denoted by $\dom(f)$: it is a convex subset of $N_\mathbb{R}$. A concave function is said to be \emph{closed} if it is upper semicontinuous. Every concave function with closed effective domain and continuous on it is closed. The \emph{recession function} of a closed concave function $f$ is the concave conical function which takes on $u\in N_\mathbb{R}$ the value \[\rec(f)(u):=\lim_{\lambda\to\infty}\frac{f(v_0+\lambda u)}{\lambda}\in\mathbb{R}\cup\{-\infty\},\] for any $v_0\in\dom(f)$, see \cite[Theorem 8.5]{R}. Finally, a concave function $f$ with effective domain a polyhedron in $N_\mathbb{R}$ is \emph{piecewise affine} if \[f(u)=\min_{\alpha\in S}(\langle\alpha,u\rangle+c_\alpha)\] for every $u\in\dom(f)$, with $S$ a finite subset of $M_\mathbb{R}$ and $c_\alpha\in\mathbb{R}$ for every $\alpha\in S$. 
\\To each concave function $f$ on $N_\mathbb{R}$, one can associate its \emph{Legendre-Fenchel dual}, which is the closed concave function $f^\vee$ on $M_\mathbb{R}$ defined as \[f^\vee(x):=\inf_{u\in N_\mathbb{R}}(\langle x,u\rangle-f(u)),\]for every $x\in M_\mathbb{R}$, see \cite[\S 2.2]{BPS}. If $f$ is closed, $(f^\vee)^\vee=f$. The effective domain of $f^\vee$ is a convex subset of $M_\mathbb{R}$, which one calls the \emph{stability set of $f$} and denotes by $\stab(f)$.
\\The following example is classical and will play a role later on.

\begin{ex}\label{indicator and support function}
Any nonempty convex body $B$ in $M_\mathbb{R}$ induces a concave function $\Psi_B$ on $N_\mathbb{R}$, called the \emph{support function} of $B$ and defined as \[\Psi_B(u):=\min_{x\in B}\langle x,u\rangle\]for every $u\in N_\mathbb{R}$. Its Legendre-Fenchel dual is the \emph{indicator function $\iota_B$ of $B$}, which is the function taking the value $0$ on $B$ and $-\infty$ elsewhere. Hence, $\dom(\Psi_B)=N_\mathbb{R}$ and $\stab(\Psi_B)=B$. Notice that, whenever $B$ is a polytope, $\Psi_B$ is a conic piecewise affine concave function.
\end{ex}

We also recall that there exists a number of operations that one can define on concave functions, in addition to the usual pointwise sum and scalar multiplication. Among these, the \emph{sup-convolution} of two concave functions $f$ and $g$ on $N_\mathbb{R}$ with non-disjoint stability sets is defined as
\[(f\boxplus g)(v):=\sup_{u_1+u_2=v}(f(u_1)+g(u_2))\]
and the \emph{right scalar multiplication} of $f$ by $\lambda\in\mathbb{R}_{\geq0}$ as
\[(f\lambda)(u):=\lambda f(u/\lambda).\]
Also, the \emph{translate} of a concave function $f$ on $N_\mathbb{R}$ by a point $u_0\in N_\mathbb{R}$ is set to be \[(\tau_{u_0}f)(u):=f(u-u_0).\]These operations are dual, via Legendre-Fenchel duality, to the usual pointwise addition, scalar multiplication and sum by a linear function, respectively, see \cite[Proposition 2.3.1 and Proposition 2.3.3]{BPS}.

\subsection{Real Monge-Amp\`ere measures}\label{subsection about real Monge-Ampere measures}

For any closed concave function $f$ on $N_\mathbb{R}$ with $\dom(f)=N_\mathbb{R}$ and for any Haar measure $\mu$ on $M_\mathbb{R}$, one can define a corresponding \emph{real Monge-Amp\`ere measure} $\mathcal{M}_\mu(f)$, as in \cite[\S2.7]{BPS}. It is a measure on $N_\mathbb{R}$, of total mass $\mu(\stab(f))$, being supported on finitely many points if $f$ is piecewise affine. The Monge-Amp\`ere operator, associating to each closed concave function $f$ with $\dom(f)=N_\mathbb{R}$ the corresponding measure $\mathcal{M}_\mu(f)$ on $N_\mathbb{R}$ is homogeneous of degree $n$ with respect to pointwise scalar multiplication. It was shown in \cite{PR} that such an operator admits a polarization: for $f_1,\dots,f_n$ closed concave functions with effective domain $N_\mathbb{R}$, their \emph{mixed real Monge-Amp\`ere measure} is defined as the measure
\begin{equation}\label{mixed Monge-Ampere measure}
\MM_\mu(f_1,\dots,f_n):=\sum_{k=1}^n(-1)^{n-k}\sum_{1\leq i_1<\dots<i_k\leq n}\mathcal{M}_\mu(f_{i_1}+\dots+f_{i_k}).
\end{equation}
Notice that this definition differs from \cite[formula (14)]{PR} and \cite[Definition 2.7.12]{BPS} by a multiplicative constant. It follows from \cite[\S5]{PR} that the measure in \eqref{mixed Monge-Ampere measure} is in fact a positive measure. The so obtained mixed Monge-Amp\`ere operator is, by definition, symmetric and multilinear in its entries (with respect to pointwise sum) and it satisfies \[\MM_\mu(f,\dots,f)=n!\mathcal{M}_\mu(f)\] for every closed concave function $f$. In particular, if $f_1,\dots,f_n$ are closed concave functions with convex bodies as stability sets, their mixed real Monge-Amp\`ere measure is a finite measure on $N_\mathbb{R}$ of total mass $\MV_\mu(\stab(f_1),\dots,\stab(f_n))$. Here, $\MV_\mu$ denotes the mixed volume of convex bodies with respect to the measure $\mu$, normalized in such a way that $\MV_\mu(Q,\dots,Q)=n!\mu(Q)$ for every convex body $Q$ in $M_\mathbb{R}$.

\begin{rem}\label{lattice volume}
The integral structure on $M_\mathbb{R}$ coming from the subjacent lattice $M$ gives a distinguished measure $\vol_M$ on $M_\mathbb{R}$, which is the unique Haar measure for which the volume of a fundamental domain of $M$ is $1$ (this does not depend on the choice of a basis of $M$). To lighten the notation, the corresponding (mixed) real Monge-Amp\`ere operator will be denoted by $\MM_M$.
\end{rem}

\subsection{Mixed integrals}\label{subsection about mixed integrals}

For any Haar measure $\mu$ on $M_\mathbb{R}$, the application mapping a compactly supported concave function $g$ on $M_\mathbb{R}$ to its Lebesgue integral with respect to $\mu$ is homogeneous of degree $(n+1)$ with respect to right scalar multiplication. As a consequence of \cite[Proposition 4.5]{PS1}, there exists a polarized operator: for $g_0,\dots,g_n$, concave functions on $M_\mathbb{R}$ with respective domains the convex bodies $Q_0,\dots,Q_n$, their \emph{mixed integral} with respect to $\mu$ is defined as \[\MI_\mu(g_0,\dots,g_n):=\sum_{k=0}^n(-1)^{n-k}\sum_{0\leq i_0<\dots<i_k\leq n}\int_{Q_{i_0}+\dots+Q_{i_k}}(g_{i_0}\boxplus\dots\boxplus g_{i_k})\ d\mu.\]
This notion was introduced and studied in \cite{PS1} and \cite{PS2}. The mixed integral operator is symmetric and multilinear in its entries (with respect to sup-convolution) and it satisfies \[\MI_\mu(g,\dots,g)=(n+1)!\int_Q g\ d\mu\] for every concave function $g$ on a convex body $Q$. As for the mixed Monge-Amp\`ere operator, we will denote by $\MI_M$ the mixed integral computed with respect to the Haar measure $\vol_M$ on $M_\mathbb{R}$.

\vspace{\baselineskip}
The following proposition describes the behaviour of mixed integrals with respect to translation of the entries.

\begin{proposition}\label{translation in mixed integral does not change the value}
Let $g_i$ be a concave function defined on a convex body $Q_i$ in $M_\mathbb{R}$, for $i=0,\dots,n$. Let also $x_0\in M_\mathbb{R}$. Then, \[\MI_\mu(\tau_{x_0}g_0,g_1,\dots,g_n)=\MI_\mu(g_0,\dots,g_n).\]
\end{proposition}
\begin{proof}
For any subset $\{i_1,\dots,i_r\}\subseteq\{1,\dots,n\}$ one has, directly by definition, that\[((\tau_{x_0}g_0)\boxplus g_{i_1}\boxplus\dots\boxplus g_{i_r})(x+x_0)=(g_0\boxplus g_{i_1}\boxplus\dots\boxplus g_{i_r})(x)\]for every $x\in M_\mathbb{R}$. The change of variables formula implies then that the integrals appearing in the definitions of $\MI_\mu(\tau_{x_0}g_0,g_1,\dots,g_n)$ and of $\MI_\mu(g_0,\dots,g_n)$ are pairwise equal, from which the statement follows trivially.
\end{proof}

We now focus on a recursive formula for mixed integrals. For any closed convex subset $C$ in $M_\mathbb{R}$, and for every $u\in N_\mathbb{R}$, one can consider the subset \[C^u:=\bigg\{x\in C:\langle x,u\rangle=\inf_{y\in C}\langle y,u\rangle\bigg\}\]of $C$. For $u\neq0$, this subset is contained in an affine subspace of $M_\mathbb{R}$ of codimension 1 and parallel to $u^\perp:=\{x\in M_\mathbb{R}:\langle x,u\rangle=0\}$. It is immediate from the definition that for every pair of convex bodies $C_1$ and $C_2$ in $M_\mathbb{R}$, and for every $u\in N_\mathbb{R}$, $C_1^u+C_2^u=(C_1+C_2)^u$, where the plus sign denotes the Minkowski sum of convex sets. When $u\in N_\mathbb{Q}\setminus\{0\}$, the intersection $M(u):=M\cap u^\perp$ is a lattice of rank $n-1$ spanning the linear space $u^\perp$ and hence it induces a normalized Haar measure $\vol_{M(u)}$ on $u^\perp$, as in \hyperref[lattice volume]{Remark \ref*{lattice volume}}.

\begin{rem}
Let $B_1,\dots, B_n$ be $n$ convex bodies in $M_\mathbb{R}$ and let $u\in N_\mathbb{R}$. The invariancy under translation of $\vol_{M(u)}$ allows to consistently consider the mixed volume of $B_1^u,\dots,B_n^u$ as convex bodies in $u^\perp$.
\\Similarly, let $g_0,\dots,g_n$ be concave functions defined on convex bodies $B_0,\dots,B_n$ in $M_\mathbb{R}$, respectively. For every $u\in N_\mathbb{Q}\setminus\{0\}$, \hyperref[translation in mixed integral does not change the value]{Proposition \ref*{translation in mixed integral does not change the value}} allows to consider the mixed integral with respect to $\vol_{M(u)}$ of $g_0|_{B_0^u},\dots,g_n|_{B_n^u}$, as functions defined on convex subsets of $u^\perp$.
\end{rem}

For a concave function $g$ with effective domain a polytope $Q$ in $M_\mathbb{R}$, its \emph{hypograph} is the closed convex set \[\Gamma(g):=\{(x,t):x\in Q, t\leq g(x)\}\subseteq M_\mathbb{R}\times\mathbb{R}.\]
With the pairing between $N_\mathbb{R}\times\mathbb{R}$ and $M_\mathbb{R}\times\mathbb{R}$ given by $\langle(x,t),(u,\lambda)\rangle:=\langle x,u\rangle+t\lambda$ for every $(u,\lambda)\in N_\mathbb{R}\times\mathbb{R}$ and $(x,t)\in M_\mathbb{R}\times\mathbb{R}$, one can consider the subset $\Gamma(g)^{(u,\lambda)}$ of $\Gamma(g)$ for every $(u,\lambda)\in N_\mathbb{R}\times\mathbb{R}$. It is empty when $\lambda>0$.
\\The mixed real Monge-Amp\`ere measure of piecewise affine concave functions can be made explicit in terms of the hypographs of their Legendre-Fenchel duals. We denote by $\delta_v$ the Dirac measure supported on $v$.

\begin{proposition}\label{real Monge-Ampere measure of piecewise affine functions}
For $i=1,\dots,n$, let $g_i$ be a piecewise affine concave function with effective domain a polytope $Q_i\subset M_\mathbb{R}$, $\Gamma_i$ the hypograph of $g_i$. Denote by $\pi:M_\mathbb{R}\times\mathbb{R}\to M_\mathbb{R}$ the projection onto the first factor. For any choice of a Haar measure $\mu$ on $M_\mathbb{R}$ one has \[\MM_\mu\big(g_1^\vee,\dots,g_n^\vee\big)=\sum_{v\in N_\mathbb{R}}\MV_\mu\bigg(\pi\Big(\Gamma_1^{(v,-1)}\Big),\dots,\pi\Big(\Gamma_n^{(v,-1)}\Big)\bigg)\ \delta_v,\] and the sum is finite.
\end{proposition}
\begin{proof}
For a piecewise affine concave function $g$ with bounded domain in $M_\mathbb{R}$, its Legendre-Fenchel dual is a piecewise affine concave function with domain $N_\mathbb{R}$ and \cite[Proposition 2.7.4]{BPS} affirms that \[\mathcal{M}_\mu\big(g^\vee\big)=\sum_{v\in N_\mathbb{R}}\mu(v^*)\delta_v,\]with $v^*=\{x\in M_\mathbb{R}: g^\vee(v)=\langle x,v\rangle-g(x)\}$. From the definition of the Legendre-Fenchel duality, one has hence that
\begin{equation*}
\begin{split}
v^*&=\bigg\{x\in M_\mathbb{R}:\langle x,v\rangle-g(x)=\min_{y\in M_\mathbb{R}}(\langle y,v\rangle-g(y))\bigg\}\\&=\Big\{x\in M_\mathbb{R}:(x,g(x))\in\Gamma(g)^{(v,-1)}\Big\},
\end{split}
\end{equation*}
and so \[\mathcal{M}_\mu\big(g^\vee\big)=\sum_{v\in N_\mathbb{R}}\mu\Big(\pi\Big(\Gamma(g)^{(v,-1)}\Big)\Big)\delta_v.\]
The sum is moreover supported on finitely many $v\in N_\mathbb{R}$, corresponding to the directions of the finitely many exposed faces of $\Gamma(g)$.
\\By \cite[\S2]{AW}, the relation \[\Gamma(g_i\boxplus g_j)=\Gamma(g_i)+\Gamma(g_j)\] on the hypographs of $g_i$ and $g_j$ holds for any $i,j\in\{1,\dots,n\}$. As a consequence, for every subset $\{i_1,\dots,i_k\}\subseteq\{1,\dots,n\}$, \cite[Proposition 2.3.1]{BPS} and the linearity of $\pi$ yield
\begin{equation*}
\begin{split}
\mathcal{M}_\mu\big(g_{i_1}^\vee+\dots+g_{i_k}^\vee\big)&=\mathcal{M}_\mu\big((g_{i_1}\boxplus\dots\boxplus g_{i_k})^\vee\big)\\&=\sum_{v\in N_\mathbb{R}}\mu\Big(\pi\Big(\Gamma_{i_1}^{(v,-1)}+\dots+\Gamma_{i_k}^{(v,-1)}\Big)\Big)\delta_v\\&=\sum_{v\in N_\mathbb{R}}\mu\Big(\pi\Big(\Gamma_{i_1}^{(v,-1)}\Big)+\dots+\pi\Big(\Gamma_{i_k}^{(v,-1)}\Big)\Big)\delta_v,
\end{split}
\end{equation*}
and the sum is finite. The statement follows then from the definition of the mixed real Monge-Amp\`ere measure, rearranging the terms.
\end{proof}

We can now prove a recursive formula relating the notions of the mixed real Monge-Amp\`ere measure and the mixed integral of concave functions, via Legendre-Fenchel duality. A vector $u\in N$ is said to be \emph{primitive} if it is nonzero and there is no other element $u^\prime\in N$ such that $ku^\prime=u$ for some positive integer $k$.

\begin{theorem}\label{recursive formula with mixed integral and mixed real Monge-Ampere measure}
For $i=0,\dots,n$, let $g_i$ be a continuous concave function on a rational polytope $Q_i$ in $M_\mathbb{R}$. Then 
\begin{multline*}
\MI_M(g_0,\dots,g_n)=-\sum_{\substack{u\in N\\\primitive}}\Psi_{Q_0}(u)\MI_{M(u)}\big(g_1|_{Q_1^{u}},\dots,g_n|_{Q_n^{u}}\big)\\-\int_{N_\mathbb{R}}g_0^\vee\ d\MM_{M}\big(g_1^\vee,\dots,g_n^\vee\big),
\end{multline*}
the first sum being finite.
\\In particular, if $g_i$ is a piecewise affine concave function on $Q_i$ with hypograph $\Gamma_i$ for any $i=0,\dots,n$, denoting by $\pi:M_\mathbb{R}\times\mathbb{R}\to M_\mathbb{R}$ the projection onto the first factor, one has
\begin{multline*}
\MI_M(g_0,\dots,g_n)=-\sum_{\substack{u\in N\\\primitive}}\Psi_{Q_0}(u)\MI_{M(u)}\big(g_1|_{Q_1^{u}},\dots,g_n|_{Q_n^{u}}\big)\\-\sum_{v\in N_\mathbb{R}}g_0^\vee(v)\MV_{M}\bigg(\pi\Big(\Gamma_1^{(v,-1)}\Big),\dots,\pi\Big(\Gamma_n^{(v,-1)}\Big)\bigg).
\end{multline*}
\end{theorem}
\begin{proof}
By \cite[Proposition 2.5.23 (1)]{BPS}, any continuous concave function on a polytope can be approximated, with respect to uniform convergence, by a sequence of piecewise affine concave functions on the polytope itself. On the other hand, the Legendre-Fenchel duality and the real Monge-Amp\`ere operator are continuous with respect to uniform limits of concave functions, see \cite[Proposition 2.2.3]{BPS} and \cite[\S3]{RT}, respectively. It is not difficult to show that the same holds for mixed integrals. Thanks to \hyperref[real Monge-Ampere measure of piecewise affine functions]{Proposition \ref*{real Monge-Ampere measure of piecewise affine functions}}, it is hence enough to prove the formula in the particular case of $g_0,\dots,g_n$ being piecewise affine concave functions.
\\Let hence $g_i$ be a concave piecewise affine function on the rational polytope $Q_i$ in $M_\mathbb{R}$, $\Gamma_i$ its hypograph, for $i=0,\dots,n$. The choice of a basis of $N$ (and of the dual basis of $M$) endows $N_\mathbb{R}$ and $M_\mathbb{R}$ with an euclidean structure, allowing to consider the sets \[\mathbb{S}^{n-1}:=\{w\in N_\mathbb{R}:\|w\|=1\}\subseteq N_\mathbb{R}\]and\[\mathbb{S}^n_-:=\{(v,t)\in N_\mathbb{R}\times\mathbb{R}:\|(v,t)\|=1,t<0\}\subseteq N_\mathbb{R}\times\mathbb{R}.\] After a change of sign due to the use of a different notation, \cite[Proposition 8.5]{PS2} affirms that
\begin{multline}\label{equation in the proof of the recursive formula for mixed integrals}
\MI_M(g_0,\dots,g_n)=-\sum_{w\in\mathbb{S}^{n-1}}\Psi_{Q_0}(w)\MI_{n-1}\big(g_1|_{Q_1^w},\dots,g_n|_{Q_n^w}\big)\\-\sum_{r\in\mathbb{S}^n_-}\Psi_{\Gamma_0}(r)\ \MV_n(\Gamma^r_1,\dots,\Gamma^r_n),
\end{multline}
where, on the right hand side, one refers to the mixed integral with respect to the measure obtained restricting $\vol_M$ to $w^\perp$ and to the mixed volume with respect to the restriction of $\vol_{M\oplus\mathbb{Z}}$ to $r^\perp$.
\\Concerning the first sum on the right hand side of \eqref{equation in the proof of the recursive formula for mixed integrals}, if a term in the sum is different from zero, then there exists a subset $I\subset\{1,\dots,n\}$ such that the Minkowski sum of $Q_i^w$, with $i\in I$, is of dimension $n-1$; in particular, denoting $Q:=Q_1+\dots+Q_n$, $Q^w=Q_1^w+\dots+Q_n^w$ needs to be of dimension $n-1$. As a consequence, one can restrict the sum to the set of vectors $w\in\mathbb{S}^{n-1}$ for which $Q^w$ is a $(n-1)$-dimensional face of $Q$. This set is included in the set of vectors of unitary length which are perpendicular to a $(n-1)$-dimensional face of $Q$, hence it is finite since $Q$ is a polytope. Moreover, since $Q$ is rational, the ray spanned by such a vector $w$ contains a unique primitive vector $u\in N$. The linearity of $\Psi_{Q_0}$ yields hence the equality
\[\sum_{w\in\mathbb{S}^{n-1}}\Psi_{Q_0}(w)\MI_{n-1}\big(g_1|_{Q_1^w},\dots,g_n|_{Q_n^w}\big)=\sum_{\substack{u\in N\\\primitive}}\frac{\Psi_{Q_0}(u)}{\|u\|}\MI_{n-1}\big(g_1|_{Q_1^u},\dots,g_n|_{Q_n^u}\big).
\]
The fact that the restriction of $\vol_M$ to $u^\perp$ is equal to the measure $\vol_{M(u)}$ multiplied by $\|u\|$, see \cite[proof of Corollary 2.7.10]{BPS}, allows to conclude that the first sum in \eqref{equation in the proof of the recursive formula for mixed integrals} coincides with the desired one.
\\Regarding the second sum in \eqref{equation in the proof of the recursive formula for mixed integrals}, there exists an obvious bijection between $\mathbb{S}^n_-$ and $N_\mathbb{R}$ given by associating to each $r\in\mathbb{S}^n_-$ the only vector $v\in N_\mathbb{R}$ such that $(v,-1)$ lies on the line spanned by $r$. Hence,
\[\sum_{r\in\mathbb{S}^n_-}\Psi_{\Gamma_0}(r)\ \MV_n(\Gamma^r_1,\dots,\Gamma^r_n)=\sum_{v\in N_\mathbb{R}}\frac{\Psi_{\Gamma_0}(v,-1)}{\|(v,-1)\|}\ \MV_n\Big(\Gamma^{(v,-1)}_1,\dots,\Gamma^{(v,-1)}_n\Big).\]Directly by the definition of the Legendre-Fenchel duality, one has that $\Psi_{\Gamma_0}(v,-1)=g_0^\vee(v)$. The statement follows then from the fact that for every Borel set $E$ in $(v,-1)^\perp$, the measure of $E$ with respect to the restriction of $\vol_{M\oplus\mathbb{Z}}$ to $(v,-1)^\perp$ equals $\|(v,-1)\|\cdot\vol_M(\pi(E))$, again by \cite[proof of Corollary 2.7.10]{BPS}.
\end{proof}

\begin{rem}\label{recursive formula with sum over the faces}
For a rational polytope $P$ of full dimension $n$ in $M_\mathbb{R}$, every facet $F$ of $P$, that is a face of dimension $n-1$, admits a distinguished orthogonal vector: it is the unique primitive vector $v_F\in N$ which satisfies $P^{v_F}=F$. Under the additional assumption that the Minkowski sum $Q:=Q_1+\dots+Q_n$ is of dimension $n$ in $M_\mathbb{R}$, the formula in \hyperref[recursive formula with mixed integral and mixed real Monge-Ampere measure]{Theorem \ref*{recursive formula with mixed integral and mixed real Monge-Ampere measure}} can be written as
\begin{multline*}
\MI_M(g_0,\dots,g_n)=-\sum_{F}\Psi_{Q_0}(v_F)\MI_{M(v_F)}\big(g_1|_{Q_1^{v_F}},\dots,g_n|_{Q_n^{v_F}}\big)\\-\int_{N_\mathbb{R}}g_0^\vee\ d\MM_{M}\big(g_1^\vee,\dots,g_n^\vee\big),
\end{multline*}
the first sum being over the finite set of facets of the polytope $Q$. Indeed, in such a situation the application $F\mapsto v_F$ realizes a bijection between the set of facets of $Q$ and the set of primitive vectors $u\in N$ for which $Q^u$ is a $(n-1)$-dimensional face of $Q$, which are the only vectors for which the term of the sum in the statement of the theorem does not vanish.
\end{rem}

\begin{rem}\label{remark for the Legendre-Fenchel dual version of the recursive formula for mixed integrals}
The statement of \hyperref[recursive formula with mixed integral and mixed real Monge-Ampere measure]{Theorem \ref*{recursive formula with mixed integral and mixed real Monge-Ampere measure}} can be reformulated in terms of Legendre-Fenchel duality. For $i=0,\dots,n$, let $f_i$ be a concave function on $N_\mathbb{R}$ with stability set a rational polytope $Q_i$ in $M_\mathbb{R}$. Under the assumption that $Q_1+\dots+Q_n$ is of dimension $n$ in $M_\mathbb{R}$, \hyperref[recursive formula with sum over the faces]{Remark \ref*{recursive formula with sum over the faces}} yields
\begin{multline}\label{version of recursive formula in terms of Legendre-Fenchel duals}
\MI_M(f^\vee_0,\dots,f^\vee_n)=-\sum_{F}\Psi_{Q_0}(v_F)\MI_{M(v_F)}\big(f^\vee_1|_{Q_1^{v_F}},\dots,f^\vee_n|_{Q_n^{v_F}}\big)\\-\int_{N_\mathbb{R}}f_0\ d\MM_{M}\big(f_1,\dots,f_n\big).
\end{multline}
Indeed, it is sufficient to readily apply the previous theorem to the functions $f_0^\vee,\dots,f_n^\vee$, which are continuous on their domain and satisfy the equality $(f_i^\vee)^\vee=f_i$ for each $i=0,\dots,n$ by concavity and closedness. It is easy to verify that the choice $f_0=\dots=f_n=f$ in \eqref{version of recursive formula in terms of Legendre-Fenchel duals} yields the formula in \cite[Corollary 2.7.10]{BPS}.
\end{rem}

We present now two applications of the recursive formula proved above. The first one concerns the computation of the mixed integral when all except one entry are indicator functions in the sense of \hyperref[indicator and support function]{Example \ref*{indicator and support function}}.

\begin{corollary}\label{mixed integral with indicator functions}
Let $Q_1,\dots,Q_n$ be rational polytopes in $M_\mathbb{R}$ and $f$ a concave function on $N_\mathbb{R}$ with stability set a rational polytope. Then
\[\MI_M\big(\iota_{Q_1},\dots,\iota_{Q_n},f^\vee\big)=-\MV_M(Q_1,\dots,Q_n)\cdot f(0).\]
\end{corollary}
\begin{proof}
By symmetry, one can develop the recursive formula in \hyperref[remark for the Legendre-Fenchel dual version of the recursive formula for mixed integrals]{Remark \ref*{remark for the Legendre-Fenchel dual version of the recursive formula for mixed integrals}} with respect to $f^\vee$ to obtain
\[\MI_M\big(\iota_{Q_1},\dots,\iota_{Q_n},f^\vee\big)=-\int_{N_\mathbb{R}}f\ d\MM_M\Big(\iota_{Q_1}^\vee,\dots,\iota_{Q_n}^\vee\Big),\]
the indicator functions $\iota_{Q_1},\dots,\iota_{Q_n}$ being zero where defined. The duality in \hyperref[indicator and support function]{Example \ref*{indicator and support function}} and the fact that
\[\MM_M\big(\Psi_{Q_1},\dots,\Psi_{Q_n}\big)=\MV_M(Q_1,\dots,Q_n)\delta_0\]
because of \hyperref[real Monge-Ampere measure of piecewise affine functions]{Proposition \ref*{real Monge-Ampere measure of piecewise affine functions}} conclude the proof.
\end{proof}

The second application explains how the mixed integral behaves with respect to pointwise sum by a constant in one entry.

\begin{corollary}\label{mixed integral with pointwise sum by a constant}
Let $g_i$ be a concave function defined on a rational polytope $Q_i\subseteq M_\mathbb{R}$, for $i=0,\dots,n$, and $c\in\mathbb{R}$. Then
\[\MI_M(g_0,\dots,g_{n-1},g_n+c)=\MI_M(g_0,\dots,g_n)+c\cdot\MV_M(Q_0,\dots,Q_{n-1}).\]
\end{corollary}
\begin{proof}
Denoting by $c\delta_0$ the concave function which has value $c$ at $0$ and $-\infty$ otherwise, it follows from the definitions that $g_n+c=g_n\boxplus c\delta_0$. The multilinearity of mixed integrals implies then that \[\MI_M(g_0,\dots,g_{n-1},g_n+c)=\MI_M(g_0,\dots,g_n)+\MI_M(g_0,\dots,g_{n-1},c\delta_0).\]
Using the fact that $(c\delta_0)^\vee=-c$, the recursive formula in \hyperref[recursive formula with mixed integral and mixed real Monge-Ampere measure]{Theorem \ref*{recursive formula with mixed integral and mixed real Monge-Ampere measure}}, developed with respect to $c\delta_0$, yields
\[\MI_M(g_0,\dots,g_{n-1},c\delta_0)=\int_{N_\mathbb{R}}c\ d\MM_M(g_0^\vee,\dots,g_{n-1}^\vee)=c\cdot\MM_M(g_0^\vee,\dots,g_{n-1}^\vee)(N_\mathbb{R}).\]
The statement follows then from the fact that the total volume of the mixed Monge-Amp\`ere measure of $g_0^\vee,\dots,g_{n-1}^\vee$ is equal to $\MV_M(\dom(g_0),\dots,\dom(g_{n-1}))$ by \cite[Proposition 3(iv)]{PR}.
\end{proof}

We conclude the section by proving a formula expressing the mixed integral of a $(n+1)$-tuple of concave functions on $M_\mathbb{R}$ where one of them is the indicator function of a line segment. \\Let $m$ be a primitive vector of $M$ and consider the quotient $P:=M/\mathbb{Z}m$. Since $m$ is primitive, $P$ is a lattice of rank $n-1$. By abuse of notation, let $\pi$ denote both the projection from $M$ to $P$ and the induced linear map from $M_\mathbb{R}$ to $P_\mathbb{R}$. For each closed concave function $g$ defined on a compact subset $B$ of $M_\mathbb{R}$, let
\begin{equation}\label{definition of direct image of concave functions}
\pi_*g:\pi(B)\to\mathbb{R},\quad x\mapsto\max_{y\in \pi^{-1}(x)}g(y)
\end{equation}
be the \emph{direct image} of $g$ by $\pi$. It is a well defined closed concave function with domain a bounded subset of $P_\mathbb{R}$, see \cite[Theorem 5.7 and Theorem 9.2]{R}. Finally, for $x_1,x_2\in M_\mathbb{R}$, denote by $\overline{x_1x_2}$ the line segment in $M_\mathbb{R}$ with extremal points $x_1$ and $x_2$. The following lemma is a generalization of \cite[exercise 3 at page 128]{Ewald} and seems to be well-known to experts, though we could not find an adequate reference in the literature for its proof.

\begin{lemma}\label{lemma for mixed volumes and projections}
In the above hypotheses and notations and for $n\geq2$, let $Q_1,\dots,Q_{n-1}$ be polytopes in $M_\mathbb{R}$. Then,
\[\MV_M\big(\overline{0m},Q_1,\dots,Q_{n-1}\big)=\MV_P\big(\pi(Q_1),\dots,\pi(Q_{n-1})\big).\]
\end{lemma}
\begin{proof}
The vector $m$ being primitive, it can be extended to a basis of the lattice $M$, see for instance \cite[Theorem 5 at page 21]{Lekk}. We suppose fixed throughout the proof such a basis $(m_1,\dots,m_{n-1},m)$ of $M$ and the induced isomorphism $M_\mathbb{R}\simeq\mathbb{R}^n$; under this identification, the normalized volume $\vol_M$ corresponds to the Lebesgue measure $ \vol_n$ on $\mathbb{R}^n$. Since $(\pi(m_1),\dots,\pi(m_{n-1}))$ is a basis of $P$, such a lattice is isomorphic to the span of $m_1,\dots,m_{n-1}$ in $M$ and hence it is identified with the linear subspace $\mathbb{R}^{n-1}\times\{0\}$ of $\mathbb{R}^n$. Moreover, $\vol_P$ corresponds to the $(n-1)$-dimensional Lebesgue measure $\vol_{n-1}$ on $\mathbb{R}^{n-1}\times\{0\}$ and the map $\pi$ to the vertical projection.
\\The claim reduces then to the particular case of a family of polytopes $Q_1,\dots,Q_{n-1}$ in $\mathbb{R}^n$, $m=(0,\dots,0,1)$ and $\pi$ the vertical projection. Denoting by $S$ the vertical segment of unitary length and rearranging the terms in the definition of the mixed volume given for instance in \cite[Definition 2.7.14]{BPS} one obtains
\begin{multline*}
\MV_n(S,Q_1,\dots,Q_{n-1})=\\\sum_{k=1}^{n-1}(-1)^{n-1-k}\sum_{1\leq i_1<\dots<i_k\leq n-1}\big(\vol_n(S+Q_{i_1}+\dots+Q_{i_k}\big)-\vol_n(Q_{i_1}+\dots+Q_{i_k})\big)
\end{multline*}
since the $n$-dimensional volume of a line segment vanishes for $n\geq2$. To prove the claim it is hence enough to show that for each polytope $Q$ in $\mathbb{R}^n$ the equality
\[\vol_n(S+Q)-\vol_n(Q)=\vol_{n-1}(\pi(Q))\]
holds. But $Q\subset S+Q$ and the difference of their volumes coincides with the integral over $\pi(Q)=\pi(S+Q)$ of the difference between the concave functions parametrizing the roof of the polytope $S+Q$ and $Q$ respectively. Such a difference being constantly equal to $1$ on $\pi(Q)$, the claim follows from the definition of the Lebesgue integral, concluding the proof.
\end{proof}

\begin{proposition}\label{mixed integral with indicator function of a segment}
In the above hypotheses and notations, let $g_i$ be a continuous concave function defined on a polytope $Q_i$ in $M_\mathbb{R}$, for $i=1,\dots,n$. Then, \[\MI_M\big(\iota_{\overline{0m}},g_1,\dots,g_n\big)=\MI_P(\pi_*g_1,\dots,\pi_*g_n).\]
\end{proposition}
\begin{proof}
For $n=1$, the claim follows from \hyperref[mixed integral with indicator functions]{Corollary \ref*{mixed integral with indicator functions}}. Assume hence $n\geq2$. Choose for each $i=1,\dots,n$ a nonpositive real number $\gamma_i$ such that $\gamma_i\leq\min_{x\in Q_i} g_i(x)$ and consider the convex body\[Q_{g_i,\gamma_i}:=\{(x,t)\in Q_i\times\mathbb{R}: \gamma_i\leq t\leq g_i(x)\}\]in $M_\mathbb{R}\times\mathbb{R}$. The formula in \cite[Proposition IV.5 (d)]{PS1} implies that
\begin{multline*}
\MI_M\big(\iota_{\overline{0m}},g_1,\dots,g_n\big)=\MV_{M\oplus\mathbb{Z}}\Big(\overline{(0,0)(m,0)},Q_{g_1,\gamma_1},\dots,Q_{g_n,\gamma_n}\Big)\\+\sum_{i=1}^n\gamma_i\MV_M\big(\overline{0m},Q_1,\dots,Q_{i-1},Q_{i+1},\dots,Q_n\big).
\end{multline*}
By the hypotheses on $m$, $(m,0)$ is a nonzero primitive vector of the lattice $M\oplus\mathbb{Z}$. The map $\pi^\prime:=\pi\times\id_\mathbb{Z}:M\oplus\mathbb{Z}\to P\oplus\mathbb{Z}$ is a surjective group homomorphism, giving $(M\oplus\mathbb{Z})/\mathbb{Z}(m,0)\simeq P\oplus\mathbb{Z}$. By \hyperref[lemma for mixed volumes and projections]{Lemma \ref*{lemma for mixed volumes and projections}},
\begin{multline*}
\MI_M\big(\iota_{\overline{0m}},g_1,\dots,g_n\big)=\MV_{P\oplus\mathbb{Z}}\big(
\pi^\prime(Q_{g_1,\gamma_1}),\dots,\pi^\prime(Q_{g_n,\gamma_n})\big)\\+\sum_{i=1}^n\gamma_i\MV_P\big(\pi(Q_1),\dots,\pi(Q_{i-1}),\pi(Q_{i+1}),\dots,\pi(Q_n)\big).
\end{multline*}
The statement follows hence from the equality in \cite[Proposition IV.5 (d)]{PS1} applied to the concave functions $\pi_*g_1,\dots,\pi_*g_n$, the direct image of $g_i$ by $\pi$ being a concave function defined on $\pi(Q_i)$ and satisfying
\begin{equation*}
\begin{split}
\pi^\prime(Q_{g_i,\gamma_i})&=\{(\pi(y),t)\in \pi(Q_i)\times\mathbb{R}: \gamma_i\leq t\leq g_i(y)\}\\&=\big\{(x,t)\in \pi(Q_i)\times\mathbb{R}: \gamma_i\leq t\leq (\pi_*g_i)(x)\big\}
\end{split}
\end{equation*}
for every $i=1,\dots,n$.
\end{proof}

\section{Ronkin functions}\label{section about Ronkin functions}

Fix for the whole section an algebraically closed complete field $(K,|\cdot|)$, that is a pair of an algebraically closed field and an absolute value with respect to which the field is complete. Depending on the nature of the absolute value, $(K,|\cdot|)$ is said to be archimedean or non-archimedean. As a consequence of Gelfand-Mazur theorem, the only archimedean algebraically closed complete field is $\mathbb{C}$ endowed with a power of the usual absolute value. When the choice of the absolute value is clear from the context, an algebraically closed complete field will be simply denoted by $K$.

\subsection{Tropicalization}

Given an affine variety $X=\spec A$ over an algebraically closed complete field $K$, let $\iota:K\to A$ be the corresponding ring homomorphism. Berkovich's construction allows to define a locally ringed space $(X^{\an},\mathcal{O}_{X^{\an}})$, called the \emph{analytification} of $X$; as a set
\[X^{\an}:=\{\|\cdot\|_x \text{ multiplicative seminorm on }A: \|\iota(k)\|_x=|k| \text{ for all }k\in K\}.\]
We refer to \cite[\S1.5]{Ber} (or also to \cite[\S1.2]{BPS} for a more concise treatment) for the definition of the topology and the structure sheaf on $X^{\an}$. By a gluing argument, one can then extend the definition to any variety over $K$.

\begin{rem}
When $K=\mathbb{C}$ endowed with the usual absolute value, the Gelfand-Mazur theorem implies that the locally ringed space produced by Berkovich's construction agrees with the standard complex analytification of a variety over $\spec \mathbb{C}$.
\end{rem}

Let $\mathbb{T}$ denote a split torus over $K$ of dimension $n$ and $M$ its character lattice, in such a way that $\mathbb{T}=\spec K[M]$. A basis of the $K$-algebra $K[M]$ will be denoted, accordingly to \cite[beginning of \S1.3]{Ful}, by $(\chi^m)_{m\in M}$ and the elements of $K[M]$ will be called \emph{Laurent polynomials} over $K$. Let $N$ be the dual lattice of $M$ and denote by $N_\mathbb{R}=N\otimes\mathbb{R}$ the associated real vector space.

\begin{defn}\label{definition of tropicalization}
The \emph{tropicalization map} $\trop:\mathbb{T}^{\an}\to N_\mathbb{R}=\hom(M,\mathbb{R})$ is the application defined by \[\left(\trop(\|\cdot\|_x)\right)(m):=-\log\|\chi^m\|_x\] for every $\|\cdot\|_x\in\mathbb{T}^{\an}$.
\end{defn}

The tropicalization map turns out to be a continuous application with respect to the Berkovich topology of $\mathbb{T}^{\an}$ and the Euclidean topology of $N_\mathbb{R}$. Moreover, in the archimedean case, the choice of a basis of $M$ allows to write such map in the more familiar form \[\trop((z_1,\dots,z_n))=(-\log|z_1|,\dots,-\log|z_n|).\] 

\begin{rem}
The tropicalization map coincides with the valuation map $\val$ used by Burgos Gil, Philippon and Sombra, see \cite[Equation 4.1.2]{BPS}.
\end{rem}

When the absolute value on $K$ is non-archimedean, one can construct a suitable section of $\trop$. For each $u\in N_\mathbb{R}$, consider the map which associates to a Laurent polynomial $f=\sum c_m\chi^m$ the real value \[\|f\|_{\kappa(u)}:=\max_m |c_m|e^{-\langle m,u\rangle}.\] One can verify that $\|\cdot\|_{\kappa(u)}$ is a multiplicative seminorm on $K[M]$ extending the absolute value on $K$, and hence it corresponds to a point $\kappa(u)\in\mathbb{T}^{\an}$, which one calls the \emph{Gauss point} over $u$. The application $\kappa:N_\mathbb{R}\to\mathbb{T}^{\an}$ defined by \[\kappa:u\to\|\cdot\|_{\kappa(u)}\] is proved to be a continuous section of $\trop$, see for instance \cite[Proposition-Definition 4.2.12]{BPS}, $\kappa$ coinciding with $\theta_0\circ\mathbf{e}$ in the cited reference.

\vspace{\baselineskip}
It is easily seen that for any closed subvariety $X$ of $\mathbb{T}$, there is a natural inclusion of sets $X^{\an}\subseteq\mathbb{T}^{\an}$, making the following definition meaningful.

\begin{defn}
Let $f$ be a nonzero Laurent polynomial in $K[M]$ and $V(f)$ the associated closed subvariety of $\mathbb{T}$. The subset \[\mathcal{A}_f:=\trop\left(V(f)^{\an}\right)\subseteq N_\mathbb{R}\]is called the \emph{amoeba} of $f$.
\end{defn}

The so-defined set has been widely studied in the literature. In the archimedean case, it coincides (up to a change of sign) with the notion of amoeba studied by Gelfand, Kapranov and Zelevinsky in \cite{GKZ} and by Passare and Rullg\aa rd in \cite{PR}. In the non-archimedean case, the amoeba of a nonzero Laurent polynomial $f$ coincides with the corner locus of the associated tropical polynomial $f^{\trop}$ (see \cite[Theorem 2.1.1]{EKL}).

\subsection{Minimal boundaries}\label{section on boundaries}

To keep the treatment of the archimedean and non-archimedean settings uniform, the following notion results to be crucial.

\begin{defn}
Let $K$ be an algebraically closed complete field and $A$ a $K$-algebra. For any set $\mathcal{S}$ of multiplicative seminorms on $A$ extending the absolute value of $K$, a \emph{boundary} of $\mathcal{S}$ is a subset $\mathcal{B}\subseteq\mathcal{S}$ such that\[\max_{x\in\mathcal{S}}\|a\|_x=\max_{x\in\mathcal{B}}\|a\|_x\]for every $a\in A$. 
\end{defn}

In other words, any boundary of $\mathcal{S}$ contains the whole information about the maximal values that the seminorms in $\mathcal{S}$ can attain. As a trivial example, $\mathcal{S}$ is a boundary of $\mathcal{S}$.
\\We are especially interested in the boundary of the fibers of the tropicalization map. Keeping the notation of the previous subsection, for a point $u\in N_\mathbb{R}$, it is immediate to show that \[\trop^{-1}(u)=\left\{\|\cdot\|_x\in\mathbb{T}^{\an}:\|\chi^m\|_x=e^{-\langle m,u\rangle}\text{ for all }m\in M\right\}.\]In the complex case, the existence of a unique minimal boundary for an algebra of functions on a compact space has been proved by Shilov. The following result, which equally holds in the non-archimedean case, is well-known by experts.

\begin{proposition}\label{minimal boundary}
The set $\trop^{-1}(u)$ has a unique minimal boundary. In the archime\-dean case, it coincides with the whole $\trop^{-1}(u)$, while in the non-archimedean case it consists of the Gauss point over $u$.
\end{proposition}
\begin{proof}
In the archimedean case, after having chosen a basis of $M$, one can treat $\mathbb{T}^{\an}$ as $(\mathbb{C}^*)^n$, by identifying a point $z=(z_1,\dots,z_n)$ of $(\mathbb{C}^*)^n$ with the seminorm $\|\cdot\|_z$ defined as \[\|f\|_z=|f(z)|\] for every $f\in K[T_1,\dots,T_n]$. Assuming that the coordinates of $u$ in the basis dual to the chosen one are $(u_1,\dots,u_n)$, the set $\trop^{-1}(u)$ corresponds to the subset of $(\mathbb{C}^*)^n$ consisting of the points $(z_1,\dots,z_n)$ satisfying $-\log|z_i|=u_i$ for every $i=1,\dots,n$. A point $\widetilde{z}$ of $\trop^{-1}(u)$ is then of the form \[\widetilde{z}:=(e^{-u_1+i\theta_1},\dots,e^{-u_n+i\theta_n}),\]with $\theta_i\in[0,2\pi)$ for every $i=1,\dots,n$. It is easy to check that such a point is the only one in $\trop^{-1}(u)$ for which \[\left\|(T_1+e^{-u_1+i\theta_1})\dots(T_n+e^{-u_n+i\theta_n})\right\|_{z}\] is maximal. As a consequence, $\widetilde{z}$ must belong to any boundary of $\trop^{-1}(u)$. This proves that the only boundary of $\trop^{-1}(u)$ is the set itself.
\\In the non-archimedean case, for every $\|\cdot\|_x\in\trop^{-1}(u)$ and for every $f=\sum c_m\chi^m\in K[M]$ one has (by definition of non-archimedean absolute value) that:\[\|f\|_x\leq\max_m|c_m|\|\chi^m\|_x=\max_m|c_m|e^{-\langle m,u\rangle}.\]This trivially implies that the Gauss norm $\|f\|_{\kappa(u)}=\max_m|c_m|e^{-\langle m,u\rangle}$ is a boundary for $\trop^{-1}(u)$.
\end{proof}

For $u\in N_\mathbb{R}$, one denotes the unique minimal boundary of $\trop^{-1}(u)$ described in the previous proposition by $\mathcal{B}(u)$. If not otherwise mentioned, such a set is considered to be endowed with the topology induced from $\mathbb{T}^{\an}$. If one sets
\[
\mathcal{B}_K:=
\begin{cases}
\big(\mathbb{S}^1\big)^n  & \text{if }K\text{ is archimedean}\\
\{1\} & \text{otherwise}
\end{cases},
\]
the boundary $\mathcal{B}(u)$ is homeomorphic to the compact group $\mathcal{B}_K$ for every $u\in N_\mathbb{R}$. This allows to define the measure \[\sigma_u:=\Haar_{\mathcal{B}(u)}\] on $\trop^{-1}(u)$, which is the Haar measure on the compact group $\mathcal{B}(u)$ normalized to have total mass $1$; it is a finite measure on $\trop^{-1}(u)$, supported on $\mathcal{B}(u)$ and distributing homogeneously on this set. In the non-archimedean case it coincides with the Dirac delta at the Gauss point over $u$.
\\There exists an embedding $\iota:N_\mathbb{R}\times\mathcal{B}_K\to\mathbb{T}^{\an}$, fitting in the commutative diagram
\begin{equation}\label{diagram}
\begin{tikzcd}
N_\mathbb{R}\times\mathcal{B}_K \arrow[r,hook,"\iota"] \arrow[d] & \mathbb{T}^{\an} \arrow[dl,"\trop"]\\
N_\mathbb{R}
\end{tikzcd}
\end{equation}
with the vertical arrow being the projection onto the first factor. In the archimedean case, it is determined by the choice of a homeomorphism $\big(\mathbb{C}^*\big)^n\simeq N_\mathbb{R}\times\big(\mathbb{S}^1\big)^n$, while in the non-archimedean case it coincides with the map $(u,1)\mapsto\kappa(u)$. The image of $\iota$ is homeomorphic to $N_\mathbb{R}\times\mathcal{B}_K$ and it is a deformation retract of the analytic torus, coinciding with it if $K$ is archimedean.

\subsection{Ronkin functions}

The terminology and notations introduced in the previous subsection allow to define Ronkin functions in the archimedean and non-archimedean case simultaneously.

\begin{defn}
Let $f$ be a nonzero Laurent polynomial over $K$. The \emph{Ronkin function} of $f$ is the map $\rho_f:N_\mathbb{R}\to\mathbb{R}$ defined as \[\rho_f(u):=\int_{\trop^{-1}(u)}-\log\|f\|_x\ d\sigma_u(x)\]for every $u\in N_\mathbb{R}$.
\end{defn}

The integral in the previous definition is finite. Indeed, logarithmic singularities are integrable in the archimedean case, and the Gauss norm of a nonzero Laurent polynomial is positive in the non-archimedean case.

\begin{rem}\label{remark about known version of Ronkin functions}
In the archimedean case, $\rho_f(u)=-N_f(-u)$, where $N_f$ is the classical Ronkin function associated to a complex Laurent polynomial (refer for instance to \cite{PR}). In the non-archimedean case, it is easily checked that $\rho_f=f^{\trop}$, where the tropicalization of a Laurent polynomial $f=\sum c_m\chi^m$ is defined by \[f^{\trop}(u)=\min_m(\langle m,u\rangle-\log|c_m|)\]for every $u\in N_\mathbb{R}$.
\end{rem}

The following property of the Ronkin function follows immediately from its definition.

\begin{proposition}\label{Ronkin function of a product}
For every pair of nonzero Laurent polynomials $f$ and $g$, $\rho_{f\cdot g}=\rho_f+\rho_g$. For every $\lambda\in K$, moreover, $\rho_{\lambda\cdot f}=-\log|\lambda|+\rho_f$.
\end{proposition}

Recall that for any Laurent polynomial $f=\sum_mc_m\chi^m$, we denote by $\NP(f)$ the \emph{Newton polytope of $f$}, that is the convex hull in $M_\mathbb{R}$ of the set $\{m\in M:c_m\neq0\}$. Its \emph{support function}, as defined in \hyperref[indicator and support function]{Example \ref*{indicator and support function}}, is denoted by $\Psi_{\NP(f)}$.

\begin{proposition}\label{properties of Ronkin functions}
Let $f$ be a nonzero Laurent polynomial. Then:
\begin{enumerate}
\item $\rho_f$ is a continuous concave function on $N_\mathbb{R}$ (in particular it is closed) and it is affine on each connected component of the complement of the amoeba of $f$
\item $|\rho_f-\Psi_{\NP(f)}|$ is bounded on $N_\mathbb{R}$
\item the stability set of $\rho_f$ coincides with $\NP(f)$ and $\rec(\rho_f)=\Psi_{\NP(f)}$.
\end{enumerate}
\end{proposition}
\begin{proof}
The statements in (1) are trivial in the non-archimedean case. In the archime\-dean case, the concavity of $\rho_f$ and its affinity outside the amoeba are shown in \cite[Theorem 1]{PR}. As a consequence of concavity, $\rho_f$ is continuous on $N_\mathbb{R}$, hence closed. 
\\To prove $(2)$, suppose that $f=\sum_mc_m\chi^m$ and let $\gamma_K(f)$ be the number of nonzero coefficients of $f$ if $K$ is archimedean, $1$ otherwise. For any $x\in\trop^{-1}(u)$, the inequality $\|f\|_x\leq \gamma_K(f)\cdot\max_m(|c_m|\|\chi^m\|_x)$ implies \[-\log\|f\|_x\geq \min_m(\langle m,u\rangle-\log|c_m|)-\log \gamma_K(f)\geq \Psi_{\NP(f)}(u)-\max_m\log|c_m|-\log \gamma_K(f)\]and hence \[\rho_f(u)\geq \Psi_{\NP(f)}(u)-\log(\gamma_K(f)\max_m|c_m|)\] for every $u\in N_\mathbb{R}$. For a reverse inequality, denote by $\mathcal{V}(f)$ the set of vertices of $\NP(f)$. Then, for every $m\in\mathcal{V}(f)$ the Ronkin function of $f$ coincides with $\langle m,u\rangle-\log|c_m|$ in a nonempty open subset of $N_\mathbb{R}$ (this follows from \cite[Proposition 2]{PR} in the archimedean setting and from \cite[Corollary 2.1.2]{EKL} in the non-archimedean one). By the concavity of $\rho_f$, one deduces hence that \[\rho_f(u)\leq \Psi_{\NP(f)}(u)-\log\min_{m\in\mathcal{V}(f)}|c_m|,\] for every $u\in N_\mathbb{R}$, concluding the proof of (2).
\\The statements in (3) follows directly from (2). Indeed, since $\big|\rho_f-\Psi_{\NP(f)}\big|$ is bounded, $\stab(\rho_f)=\stab\big(\Psi_{\NP(f)}\big)=\NP(f)$; the last equality follows then from \cite[Theorem 13.3]{R}.
\end{proof}

The calculation of the Ronkin function of a nonzero Laurent polynomial is typically very difficult. Anyway, an explicit expression for it is available in the following two simple situations.

\begin{ex}\label{example Ronkin of a monomial}
It follows from the definition that the Ronkin function of the monomial $\chi^m$ coincides with the linear function $m$ on $N_\mathbb{R}$ for every $m\in M$.
\end{ex}

\begin{ex}\label{example Ronkin of a binomial}
For any $m,m^\prime\in M$ with $m\neq m^\prime$, the Ronkin function of the binomial $f=\chi^m-\chi^{m^\prime}$ coincides with the support function of the segment $\overline{mm^\prime}$, that is \[\rho_f(u)=\min(\langle m,u\rangle, \langle m^\prime,u\rangle)\]for every $u\in N_\mathbb{R}$. To prove this, remark first that one can restrict to the case of the binomial $f=\chi^m-1$ with $m\neq0$ because of \hyperref[Ronkin function of a product]{Proposition \ref*{Ronkin function of a product}} and \hyperref[example Ronkin of a monomial]{Example \ref*{example Ronkin of a monomial}} and by factoring with a monomial. In the non-archimedean case the statement follows immediately from \hyperref[remark about known version of Ronkin functions]{Remark \ref*{remark about known version of Ronkin functions}}. In the archimedean one, the choice of a basis for $M$ allows to write
\[\rho_f(u)=-\frac{1}{(2\pi)^n}\int_{\theta_1,\dots,\theta_n\in[0,2\pi]}\log\big|e^{-m_1u_1+im_1\theta_1}\cdot\dots\cdot e^{-m_nu_n+im_n\theta_n}-1\big|\ d\theta_1\dots d\theta_n,\]
with $m_1,\dots,m_n$ being the coordinates of $m$ in such a basis and $u_1,\dots,u_n$ the coordinates of $u$ in the dual one. Assuming that $m_1>0$, which is always possible since $m\neq0$, Jensen's formula yields, for every $\theta_2,\dots,\theta_n$,
\begin{equation}\label{equality in the proof of the Ronkin function of a binomial}
\int_{\theta_1\in[0,2\pi]}\log\big|e^{-m_1u_1+im_1\theta_1}\cdot\dots\cdot e^{-m_nu_n+im_n\theta_n}-1\big|\ d\theta_1=-2\pi\sum_{j=1}^k\log\frac{|\alpha_j|}{e^{-m_1u_1}}
\end{equation}
with $\alpha_1,\dots,\alpha_k$ being the zeros of the univariate polynomial \[\big(e^{-m_2u_2+im_2\theta_2}\cdot\dots\cdot e^{-m_nu_n+im_n\theta_n}\big)T-1\] lying inside the closed disk of radius $e^{-m_1u_1}$, repeated according to multiplicity. The only complex zero of the above polynomial has modulus $e^{m_2u_2+\dots+m_nu_n}$; the integral in \eqref{equality in the proof of the Ronkin function of a binomial} is then zero if $m_1u_1+\dots+m_nu_n>0$, otherwise it equals $-2\pi(m_1u_1+\dots+m_nu_n)$. It follows that
\[\rho_f(u)=\min(m_1u_1+\dots+m_nu_n,0),\]
hence the claim.
\end{ex}

\section{Heights of toric varieties}\label{section about Arakelov on toric varieties}

A well-suited framework to develop Arakelov geometry is provided by the study of varieties over adelic fields. In this setting, local and global heights of cycles of arbitrary dimension can be defined following \cite{Z}, \cite{G0} and \cite{C-L}. A more general approach involving $M$-fields has been suggested in \cite{G1}. Even if the theory is often phrased in terms of line bundles, we adopt here the equivalent point of view of divisors, which turns out to be more convenient in the toric case, see for instance \cite{BPRS}.

\subsection{Adelic fields}

By a \textit{place} on a field $K$, we mean an equivalence class of absolute values on $K$, that could be either archimedean or non-archimedean. Whenever $\mathfrak{M}$ is a collection of places on $K$, the subset of archimedean places in $\mathfrak{M}$ is denoted by $\mathfrak{M}_\infty$.

\begin{defn}\label{adelic family}
Let $K$ be a field. A family of places $\mathfrak{M}$ on $K$ is said to be \emph{adelic} if it satisfies the following properties:
\begin{enumerate}
\item for every $v\in\mathfrak{M}\setminus\mathfrak{M}_\infty$, one (and hence all) absolute value in the class of $v$ is associated to a nontrivial discrete valuation
\item for each $\alpha\in K^*$, the set of places $v$ for which $|\alpha|_v\neq1$ for any $|\cdot|_v\in v$ is finite.
\end{enumerate}
\end{defn}

It is clear that the two conditions of the previous definition do not depend on the choice of the representant of the class $v$.

\begin{defn}\label{adelic field}
An \emph{adelic field} is a field $K$ together with an adelic family of places $\mathfrak{M}$ on $K$ and a choice of an absolute value $|\cdot|_v$ and of a real positive number $n_v$ for each place $v\in\mathfrak{M}$. An adelic field $\left(K,(|\cdot|_v,n_v)_{v\in\mathfrak{M}}\right)$ is said to satisfy the \emph{product formula} if for every $\alpha\in K^*$ \[\sum_{v\in\mathfrak{M}}n_v\log|\alpha|_v=0.\]
\end{defn}

Whenever there is no ambiguity on its adelic structure, an adelic field will be simply denoted by $K$.
\\The following property is an easy, though fundamental, consequence of the definition.

\begin{lemma}\label{adelic fields have finitely many archimedean places}
Any adelic field $K$ only admits finitely many archimedean places.
\end{lemma}
\begin{proof}
Since an absolute value on $K$ is non-archimedean if and only if it is bounded on the image of $\mathbb{Z}$ in $K$, a field with positive characteristic has no archimedean absolute values. Suppose hence that $K$ has characteristic zero. In this case it contains a copy of $\mathbb{Q}$ and any archimedean absolute value $|\cdot|_v$ on $K$ restricts to an archimedean absolute value on $\mathbb{Q}$. By Ostrowski's theorem, one has $|2|_v>1$. The second axiom in \hyperref[adelic field]{Definition \ref*{adelic family}} allows hence to conclude the claim.
\end{proof}

For an adelic field $K$ and a finite field extension $F$ of $K$, there exists a canonical way of endowing $F$ with the structure of an adelic field, see \cite[Remark 2.5]{G1} and \cite[\S3]{MS} for the detailed construction. With this induced adelic structure, $F$ satisfies the product formula whenever $K$ does.

\begin{ex}\label{examples of adelic fields}
The archetypical example of an adelic field satisfying the product formula is given by the field $\mathbb{Q}$, together with the collection of all its nontrivial places, the standard normalized absolute value for each of them and weights equal to $1$.
\end{ex}

More generally, any \emph{global field}, that is a number field or the function field of a smooth projective curve over a field $k$ with the structure described in \cite[Example 1.5.4]{BPS}, is an adelic field satisfying the product formula.

\subsection{Local and global heights}

Let $K$ be an adelic field satisfying the product formula and $X$ a proper variety of dimension $n$ over $K$. For every place $v\in\mathfrak{M}$, denote by $K_v$ the completion of $K$ with respect to $|\cdot|_v$ and by $\mathbb{C}_v$ the completion of an algebraic closure of $K_v$ with respect to the unique extension of the absolute value. It is a well-known fact that $\mathbb{C}_v$ is algebraically closed; moreover, $\mathbb{C}_v$ comes with an absolute value that one denotes, with abuse of notation, by $|\cdot|_v$. The pair $(\mathbb{C}_v,|\cdot|_v)$ is hence an algebraically closed complete field as in \hyperref[section about Ronkin functions]{section \ref*{section about Ronkin functions}}. 
\\The base change $X_{\mathbb{C}_v}$ is a scheme of finite type over $\spec \mathbb{C}_v$ to which one associates its Berkovich analytification $(X_v^{\an},\mathscr{O}_{X_v^{\an}})$, whose underlying topological space is compact because of the properness of $X$. To stress its dependence on the choice of the place $v$, $X_v^{\an}$ is called the \emph{$v$-adic analytification} of $X$. Similarly, one can consider the base change $X_{K_v}$ of $X$ over $\spec K_v$ and consider its Berkovich analytification $X_{K_v}^{\an}$. The two spaces are related by the isomorphism\[X_{K_v}^{\an}\simeq X_v^{\an}/\Gal(\overline{K}_v^{\sep}/K_v),\] as shown in \cite[Proposition 1.3.5]{Ber}. Moreover, there exists a surjective morphism of locally ringed space $\pi_v:X_v^{\an}\to X_{\mathbb{C}_v}$.

\begin{rem}
By Ostrowski's and Gelfand-Mazur theorems, if $v$ is an archimedean absolute value on $K$, $\mathbb{C}_v$ is isometric to the field $\mathbb{C}$ endowed with a power of the usual absolute value. In this case, the Berkovich space $(X_v^{\an},\mathscr{O}_{X_v^{\an}})$ is isomorphic to the usual complex analytification of $X_\mathbb{C}$.
\end{rem}

For any line bundle $L$ on $X$, its $v$-adic analytification is the analytic line bundle \[L_v^{\an}:=\pi_v^*L_{\mathbb{C}_v}\] on $X_v^{\an}$. Continuous metrics on $L_v^{\an}$ are defined as in \cite[\S1.1.1]{C-L1}, unregardingly on the nature of the place $v$. Relevant classes of metrics on $L_v^{\an}$ are \emph{smooth} metrics in the archimedean case and \emph{algebraic} (or, equivalently, formal, see \cite[Proposition 8.13]{GK}) metrics when $v$ is non-archimedean, see for example \cite[\S1]{C-L1} and \cite[8.8 and 8.12]{GK} for the precise definitions. A divisor $D$ on $X$ together with a continuous $\Gal(\overline{K}_v^{\sep}/K_v)$-invariant metric $\|\cdot\|_v$ on the analytic line bundle $\mathscr{O}(D)_v^{\an}$ is called a \emph{$v$-adic metrized divisor} and it is denoted by $\overline{D}_v$ or also by $(\mathscr{O}(D),\|\cdot\|_v)$. Sums and pull-backs of $v$-adic metrized divisors can be defined as in \cite[\S1.2]{C-L1}.
\\A $v$-adic metrized divisor $\overline{D}_v$ is said to be \emph{semipositive} if the corresponding metric can be approximated by semipositive smooth (when $v$ is archimedean) or algebraic (when $v$ is non-archimedean) semipositive metrics in the sense of \cite[\S1.4]{BPS}. For any $d$-dimensional subvariety $Y$ of $X$ and for any $d$-tuple of $v$-adic semipositive metrized divisors $\overline{D}_{0,v},\dots,\overline{D}_{d-1,v}$ on $X$, there exists a positive measure 
\begin{equation}\label{measure for semipositive metrized line bundles}
c_1(\overline{D}_{0,v})\wedge\dots\wedge c_1(\overline{D}_{d-1,v})\wedge\delta_{Y}
\end{equation}
on $X_v^{\an}$, which was first introduced in \cite[D\'efinition 2.4 and Proposition 2.7 b)]{C-L} in the non-archimedean setting and extended in \cite[\S3.8]{G2} under weaker assumptions. The suggestive notation for the measure in \hyperref[measure for semipositive metrized line bundles]{(\ref*{measure for semipositive metrized line bundles})} is compatible with the wedge product of first Chern forms in the smooth archimedean case, while it is justified by the recent advances in the theory of forms and currents on Berkovich spaces otherwise, as shown in \cite[\S6.9]{CLD} and \cite[Theorem 10.5]{GK}.
\\Recall also that for a $d$-dimensional cycle $Z$ in $X$ and a family $(D_0,s_0),\dots,(D_d,s_d)$ of divisors on $X$ with rational sections of the associated line bundles, one says that $s_0,\dots,s_d$ \emph{meet $Z$ properly} if for every $J\subseteq\{0,\dots,d\}$, each irreducible component of $|Z|\cap\bigcap_{i\in J}|\divisor(s_i)|$ has dimension $d-\#J$, $|\cdot|$ denoting here the support of a cycle.

\begin{defn}\label{local height definition}
Let $Z$ be a $d$-dimensional cycle in $X$ and $\big(\overline{D}_{0,v},s_0\big),\dots,\big(\overline{D}_{d,v},s_d\big)$ a collection of $v$-adic semipositive metrized divisors on $X$ with rational sections of the corresponding line bundles, with $s_0,\dots,s_d$ meeting $Z$ properly. The \emph{$v$-adic local height of $Z$ in $X$ with respect to $\big(\overline{D}_{i,v},s_i\big)$} for $i=0,\dots,d$ is defined, linearly in its irreducible components, by the recursive formula 
\begin{multline*}
h_{\overline{D}_{0,v},\dots,\overline{D}_{d,v}}(Z;s_0,\dots,s_d):=h_{\overline{D}_{0,v},\dots,\overline{D}_{d-1,v}}(Z\cdot\divisor(s_d);s_0,\dots,s_{d-1})\\-\int_{X_v^{\an}}\log\|s_d\|_{d,v}\ c_1(\overline{D}_{0,v})\wedge\dots\wedge c_1(\overline{D}_{d-1,v})\wedge\delta_{Z},
\end{multline*}
where $\|\cdot\|_{d,v}$ denotes the metric of $\overline{D}_{d,v}$ and one sets the height of the zero cycle to be zero.
\end{defn}

The integrals appearing in the previous definition are well-defined, as shown in \cite[Th\'eor\`eme 4.1]{CLT} in both the archimedean and non-archimedean setting, and \cite[Theorem 1.4.3]{GH} for the case of non-archimedean valuations which are not necessarily discrete. The $v$-adic local height function is moreover symmetric and multilinear with respect to sums of metrized divisors with rational sections of the associated line bundles, see \cite[Proposition 3.4 and Remark 9.3]{G}.
\vspace{\baselineskip}
\\The adelic structure on the field $K$ allows to define a \emph{semipositive metrized divisor} $\overline{D}$ on $X$ by the choice, for every place $v\in\mathfrak{M}$, of a continuous semipositive metric on $\mathscr{O}(D)_v^{\an}$. This global definition induces a notion of a $v$-adic local height function at each place of $K$. Some care has to be taken when defining global heights as sums of such $v$-adic local heights, since they do not need to be well-defined in general. 

\begin{defn}\label{global heights of integrable cycles}
A $d$-dimensional irreducible subvariety $Y$ of $X$ is said to be \emph{integrable} with respect to the choice of $d+1$ semipositive metrized divisors $\overline{D}_0,\dots,\overline{D}_d$ if there exists a birational proper map $\varphi:Y^\prime\to Y$, with $Y^\prime$ projective, and sections $s_i$ of $\varphi^*\mathscr{O}(D_i)$ for each $i=0,\dots,d$, meeting $Y^\prime$ properly, such that the $v$-adic local height \[h_{\varphi^*\overline{D}_{0,v},\dots,\varphi^*\overline{D}_{d,v}}(Y^\prime;s_0,\dots,s_d)\] is zero for all but finitely many places $v\in\mathfrak{M}$. A $d$-dimensional cycle is said to be \emph{integrable} if each of its irreducible components is. If $Y$ is an integrable $d$-dimensional irreducible subvariety, the \emph{global height of $Y$ in $X$ with respect to $\overline{D}_0,\dots,\overline{D}_d$} is defined as \[h_{\overline{D}_0,\dots,\overline{D}_d}(Y):=\sum_{v\in\mathfrak{M}}n_v\ h_{\varphi^*\overline{D}_{0,v},\dots,\varphi^*\overline{D}_{d,v}}(Y^\prime;s_0,\dots,s_d).\]
The \emph{global height} of integrable cycles is defined by linearity.
\end{defn} 

The previous definition does not depend on the choice of the projective resolution $Y^\prime$ of $Y$ nor of the sections $s_0,\dots,s_d$, as a consequence of \cite[Proposition 3.6 and Remark 9.3]{G}, \cite[Theorem 1.4.17 (3)]{BPS} and the product formula on $K$. As its local counterparts, the global height is symmetric and multilinear with respect to sums of metrized divisors. Moreover, it is well-behaved under proper transformations, in the sense of the next proposition.

\begin{proposition}\label{heights and proper morphisms}
Let $\varphi:X^\prime\to X$ be a dominant morphism of proper varieties over $K$, $\overline{D}_0,\dots,\overline{D}_d$ semipositive metrized divisors over $X$ and $Z^\prime$ a $d$-dimensional cycle in $X^\prime$. The cycle $\varphi_*Z^\prime$ is integrable with respect to $\overline{D}_0,\dots,\overline{D}_d$ if and only if $Z^\prime$ is integrable with respect to $\varphi^*\overline{D}_0,\dots,\varphi^*\overline{D}_d$ and in this case \[h_{\overline{D}_0,\dots,\overline{D}_d}(\varphi_*Z^\prime)=h_{\varphi^*\overline{D}_0,\dots,\varphi^*\overline{D}_d}(Z^\prime).\]
\end{proposition}
\begin{proof}
The statement about integrability is \cite[Proposition 1.5.8 (2)]{BPS}, while the equality of the global heights follows from the same property on local heights, as proved in \cite[Proposition 3.6 and Remark 9.3]{G} in the more general context of pseudo-divisors.
\end{proof}

\subsection{Heights on toric varieties}\label{section height on toric varieties}

In this section the basic constructions and results of the arithmetic geometry of toric varieties are recalled, following the treatment of \cite{BPS}. Let hence $X_\Sigma$ be a proper toric variety of dimension $n$ over an adelic field $K$, with torus $\mathbb{T}$ and dense open orbit $X_0$. Denote by $N$ and $M$ the character and cocharacter groups of $\mathbb{T}$ and by $N_\mathbb{R}$ and $M_\mathbb{R}$ the associated (reciprocally dual) real vector spaces. The toric variety $X_\Sigma$ is associated with a complete fan $\Sigma$ in $N_\mathbb{R}$, whose collection of $k$-dimensional cones is written $\Sigma^{(k)}$. For any $\tau\in\Sigma^{(1)}$, denote by $v_\tau$ its minimal nonzero integral vector and by $V(\tau)$ its associated orbit closure, consistently with \cite[\S3.1]{Ful}.
\\Divisors on $X_\Sigma$ which are invariant under the torus action are called \emph{toric divisors} and admit a nice combinatorial description, as follows. By \cite[\S I.2, Theorem 9]{KKMS}, there exists a bijection between the set of toric Cartier divisors on $X_\Sigma$ and the set of \emph{virtual support functions} on $\Sigma$, that is piecewise linear real-valued functions on the support of $\Sigma$, with integral slope on each cone of $\Sigma$. The toric Cartier divisor constructed from the virtual support function $\Psi$ is denoted by $D_\Psi$ and it defines a distinguished rational section $s_\Psi$ of $\mathscr{O}(D_\Psi)$ satisfying $\divisor(s_\Psi)=D_\Psi$. The corresponding Weil divisor is given by
\begin{equation}\label{divisor of toric section}
[D_\Psi]=\sum_{\tau\in\Sigma^{(1)}}-\Psi(v_\tau)V(\tau).
\end{equation}
In particular, the rational section $s_\Psi$ is regular and nowhere vanishing on $X_0$. A toric divisor $D_\Psi$ also determines a polyhedron \[\Delta_\Psi:=\{x\in M_\mathbb{R}:x-\Psi\geq0\}\] in $M_\mathbb{R}$, which is bounded because of the properness of $X_\Sigma$, see \cite[Proposition at page 67]{Ful}, and coincides with the stability set of $\Psi$ if $\Psi$ is concave. Many algebro-geometric properties of a toric divisor are read from its associated virtual support function: for instance, $D_\Psi$ is generated by global sections if and only if $\Psi$ is concave, and it is ample if and only if $\Psi$ is concave and restricts to different linear functions on different maximal cones of $\Sigma$.
\\Regarding the arithmetic part of the toric dictionary, let $D_\Psi$ be a toric divisor on $X_\Sigma$, with associated virtual support function $\Psi$. The continuous metrics on $\mathscr{O}(D_\Psi)$ admitting a combinatorial description are the ones which are invariant under the action of a certain compact torus, see \cite[\S4.2]{BPS} for more details about this notion. In concrete terms, a continuous metric $\|\cdot\|_v$ on $\mathscr{O}(D_\Psi)_v^{\an}$ is called a \emph{$v$-adic toric metric} if the map
\[(X_0)_v^{\an}\to\mathbb{R},\quad p\mapsto\|s_\Psi(p)\|_v\]
is constant along the fibers of the $v$-adic tropicalization map $\trop_v:(X_0)_v^{\an}\to N_\mathbb{R}$ introduced in \hyperref[definition of tropicalization]{Definition \ref*{definition of tropicalization}}. A toric divisor $D$ together with a $v$-adic toric metric on $\mathscr{O}(D)$ is called a \emph{$v$-adic toric metrized divisor}. To a $v$-adic toric metrized divisor $\overline{D}_v$ one can associate the real-valued map $\psi_{\overline{D}_v}$ on $N_\mathbb{R}$ satisfying the equality
\begin{equation}\label{definition of the metric function}
\psi_{\overline{D}_v}\circ\trop_v=\log\|s_D\|_v
\end{equation}
on the analytic torus $(X_0)_v^{\an}$, $s_D$ being the distinguished rational section of $\mathscr{O}(D)$.
\\The map $\psi_{\overline{D}_v}$, which will be referred to as the \emph{metric function} of $\overline{D}_v$, has been introduced by Burgos Gil, Philippon and Sombra in their study of Arakelov geometry of toric varieties to encode many arithmetic properties of $\overline{D}_v$, see \cite[Chapter 4]{BPS}. For instance, it is smooth in the archimedean case if the metric is smooth, while in the non-archimedean setting it is rational piecewise affine if the metric is algebraic, see \cite[Theorem 4.5.10 (1)]{BPS} and \cite[Proposition 2.5.5]{GH}. Also, the semipositivity of $\overline{D}_v$ is translated into the concavity of its corresponding metric function.

\begin{theorem}\label{characterization of toric metrics}
Let $D$ be the toric divisor associated to the virtual support function $\Psi$. The assignment $\|\cdot\|_v\mapsto\psi_{\overline{D}_v}$ is a bijection between the space of $v$-adic semipositive toric metrics on $\mathscr{O}(D)_v^{\an}$ and the space of concave functions $\psi$ on $N_\mathbb{R}$ such that $|\psi-\Psi|$ is bounded.
\end{theorem}
\begin{proof}
This is \cite[Theorem 4.8.1 (1)]{BPS}. The extension to the general non-archime\-dean case is \cite[Theorem 2.5.8]{GH}.
\end{proof}

If $\overline{D}_v$ is a $v$-adic semipositive toric metrized divisor, the Legendre-Fenchel dual of the metric function of $\overline{D}_v$ is called the \emph{roof function} of $\overline{D}_v$ and denoted by $\vartheta_{\overline{D}_v}$: it is a concave function on $M_\mathbb{R}$ with effective domain the polytope $\Delta_\Psi$.
\\A toric divisor admits a $v$-adic semipositive metric if and only if it is generated by global sections, as proved in \cite[Corollary 4.8.5]{BPS}. For such divisors, moreover, there exists a distinguished choice of a $v$-adic semipositive metric.

\begin{defn}\label{canonical metric definition}
Let $D$ be a toric divisor generated by global sections, $\Psi$ its associated virtual support function. The \emph{$v$-adic canonical metric} on $D$ is the semipositive toric metric on $\mathscr{O}(D)_v^{\an}$ corresponding to $\Psi$ in the bijection of \hyperref[characterization of toric metrics]{Theorem \ref*{characterization of toric metrics}}.
\end{defn}

In the non-archimedean case, the canonical metric on $D$ coincides with the algebraic metric induced by the canonical model of $X_\Sigma$ and $D$, see \cite[Example 4.5.4]{BPS}.
\\For semipositive $v$-adic toric metrized divisors, the measure in \eqref{measure for semipositive metrized line bundles} can be expressed in terms of the associated metric functions. Indeed, recall from \hyperref[section on boundaries]{subsection \ref*{section on boundaries}} that there exists an embedding \[\iota_v:N_\mathbb{R}\times\mathcal{B}_{\mathbb{C}_v}\to X_{0,v}^{\an}\]which fits into the \hyperref[diagram]{commutative diagram (\ref*{diagram})}, and denote by $\Haar_{\mathcal{B}_{\mathbb{C}_v}}$ the Haar measure on $\mathcal{B}_{\mathbb{C}_v}$ normalized to have total mass $1$.

\begin{theorem}\label{Chambert-Loir measure in the toric case}
For $i=0,\dots,n-1$, let $\overline{D}_{i,v}$ be a semipositive $v$-adic toric metrized divisor on $X_\Sigma$, $\Psi_i$ the virtual support function associated to $D_i$ and $\psi_{i,v}$ the metric function of $\overline{D}_{i,v}$. Then, the positive measure \[c_1(\overline{D}_{0,v})\wedge\dots\wedge c_1(\overline{D}_{n-1,v})\wedge\delta_{X_\Sigma}\]is zero outside $X_{0,v}^{\an}$ and \[c_1(\overline{D}_{0,v})\wedge\dots\wedge c_1(\overline{D}_{n-1,v})\wedge\delta_{X_\Sigma}|_{X_{0,v}^{\an}}=(\iota_v)_*\big(\MM_M(\psi_{0,v},\dots,\psi_{n-1,v})\times\Haar_{\mathcal{B}_{\mathbb{C}_v}}\big).\]
In particular, \[(\trop_v)_*\big(c_1(\overline{D}_{0,v})\wedge\dots\wedge c_1(\overline{D}_{n-1,v})\wedge\delta_{X_\Sigma}|_{X_{0,v}^{\an}}\big)=\MM_M(\psi_{0,v},\dots,\psi_{n-1,v})\]as measures on $N_\mathbb{R}$. 
\end{theorem}
\begin{proof}
The first statement follows from \cite[Theorem 1.4.10 (1)]{BPS} and \cite[Corollary 1.4.5]{GH}. The expression for the measure in the archimedean and the discrete non-archimedean case is obtained from \cite[Theorem 4.8.11]{BPS} and multilinearity; the general non-archimedean case is deduced from \cite[Theorem 2.5.10]{GH}. The last assertion is an easy consequence of the commutativity of the \hyperref[diagram]{diagram (\ref*{diagram})}.
\end{proof}

Moving to the global case, a (\emph{semipositive}) \emph{toric metric} on a toric divisor $D$ is a choice, for each place $v\in\mathfrak{M}$, of a (semipositive) $v$-adic toric metric on the line bundle $\mathscr{O}(D)$. The toric divisor $D$ together with a (semipositive) toric metric is called a (\emph{semipositive}) \emph{toric metrized divisor} and it is denoted by $\overline{D}$. From the point of view of convex geometry, the semipositive toric metrized divisor $\overline{D}$ is completely described by the collection $(\psi_v)_{v\in\mathfrak{M}}$ of its metric functions or, equivalently, by the collection $(\vartheta_v)_{v\in\mathfrak{M}}$ of its roof functions.
\\A notion of well-behaving toric metrics was defined in \cite[Definition 4.9.1]{BPS}.

\begin{defn}\label{definition of adelic toric metric}
A toric metric $(\|\cdot\|_v)_{v\in\mathfrak{M}}$ on a toric divisor is said to be \emph{adelic} if for all but finitely many $v\in\mathfrak{M}$ the $v$-adic toric metric $\|\cdot\|_v$ is the canonical one, in the sense of \hyperref[canonical metric definition]{Definition \ref*{canonical metric definition}}.
\end{defn}

In convex terms, a toric metric on the toric divisor $D$ associated to the virtual support function $\Psi$ is adelic if and only if the family $(\psi_v)_v$ of its metric functions satisfies $\psi_v=\Psi$ for all but finitely many $v\in\mathfrak{M}$.
\\It follows from \cite[Theorem 5.2.4]{BPS} that any toric subvariety of $X_\Sigma$ is integrable with respect to the choice of adelic semipositive toric metrized divisors. In particular, one can compute the global height of the $n$-dimensional cycle $X_\Sigma$ with respect to such choices.

\begin{theorem}\label{height of a toric variety}
Let $\overline{D}_0,\dots,\overline{D}_n$ be toric divisors over $X_\Sigma$, equipped with adelic semipositive toric metrics. Then\[h_{\overline{D}_0,\dots,\overline{D}_n}(X_\Sigma)=\sum_{v\in\mathfrak{M}}n_v\MI_M(\vartheta_{0,v},\dots,\vartheta_{n,v}),\] where $\vartheta_{i,v}$ is the roof function of $\overline{D}_{i,v}$, for every $i=0,\dots,n$ and $v\in\mathfrak{M}$.
\end{theorem}
\begin{proof}
This is \cite[Theorem 5.2.5]{BPS}.
\end{proof}

\section{Divisors of rational functions}\label{section divisor of rational functions}

The present section focuses on the combinatorial description of the Weil divisor on a toric variety of the rational function coming from a Laurent polynomial. This result will be used in the proof of the main theorems in the next section.
\\To fix notation, let $X_\Sigma$ be a proper smooth toric variety of dimension $n$ over a field $K$, $M$ the character lattice of its torus $\mathbb{T}$ and $N_\mathbb{R}$ the associated dual vector space. The toric variety $X_\Sigma$ has a dense open orbit $X_0$ isomorphic to $\mathbb{T}$ and hence the function field of $X_\Sigma$ coincides with $K(M)$. In particular, any Laurent polynomial $f=\sum c_m\chi^m$ gives rise to a rational function on $X_\Sigma$, which is regular and coincides with $f$ on $X_0$. To avoid confusion, one will denote by $\tilde{f}$ such a rational function.
\\Recall from \cite[Theorem 3.1.19 (a)]{CLS} that the toric variety $X_\Sigma$ is smooth if and only if each cone of $\Sigma$ is a smooth cone, that is it is generated by a part of a basis of the lattice $N$. For each cone $\tau$ in $\Sigma$ of dimension $1$, denote by $v_\tau$ its minimal nonzero integral vector, which generates $\tau\cap N$ as a monoid. If $\sigma$ is a smooth cone of dimension $n$ in $N_\mathbb{R}$, the collection $(v_\tau)_\tau$, with $\tau$ ranging in the set of one dimensional faces of $\sigma$, is a basis of $N$ and hence gives a dual basis $(v_\tau^\vee)_\tau$ of the lattice $M$.

\begin{lemma}\label{lemma prime ideal of the orbit}
Let $\sigma$ be a strongly convex polyhedral rational cone in $N_\mathbb{R}$. For every face $\tau$ of $\sigma$ of dimension $1$, the orbit closure $V(\tau)$ in the affine toric variety $X_\sigma$ is the subvariety corresponding to the prime ideal \[\mathfrak{p}=\big(\chi^m:m\in\sigma^\vee\cap M,m\notin\tau^\perp\big)\]of $\mathscr{O}(X_\sigma)=K[\sigma^\vee\cap M]$. Moreover, if $\sigma$ is smooth and of maximal dimension in $N_\mathbb{R}$, $\mathfrak{p}$ is principal and generated by $\chi^{v_\tau^\vee}$.
\end{lemma}
\begin{proof}
Recall for example from \cite[\S 3.1]{Ful} that the orbit closure $V(\tau)$ is the toric variety $\spec K[\sigma^\vee\cap\tau^\perp\cap M]$ and can be embedded in $X_\sigma=\spec K[\sigma^\vee\cap M]$ via the surjection of rings sending $\chi^m$ to itself if $m\in\tau^\perp$, and to $0$ otherwise. Then $V(\tau)$ is seen as the subvariety of $X_\sigma$ corresponding to the kernel of such homomorphism, that is
\[\mathfrak{p}=\bigoplus_{\substack{m\in\sigma^\vee\cap M\\m\notin\tau^\perp}} K\chi^m=\big(\chi^m:m\in \sigma^\vee\cap M,m\notin\tau^\perp\big),\]
proving the first statement.
\\Suppose now that $\sigma$ is smooth and of dimension $n$ in $N_\mathbb{R}$; denote by $v_1,v_2,\dots,v_n$ the basis of $N$ given by the minimal integral vectors of the rays of $\sigma$, with the assumption that $v_1=v_\tau$. By definition, \[\sigma=\mathbb{R}_{\geq0}v_1+\dots+\mathbb{R}_{\geq0}v_n.\]As a result, denoting by $(v_i^\vee)_{i=1,\dots,n}$ the basis of $M$ dual to $(v_i)_{i=1,\dots,n}$, one has that \[\langle v_i^\vee,u\rangle=\lambda_i\geq0\]for every $i\in\{1,\dots,n\}$ and for every $u=\sum_i\lambda_iv_i\in\sigma$. In particular, $v_\tau^\vee\in\sigma^\vee$. It is easy to check that $v_\tau^\vee$ is integrally valued on each element of $N$ and hence it belongs to $M$. It follows then from $\langle v_\tau^\vee,v_\tau\rangle=1$ that \[\Big(\chi^{v_\tau^\vee}\Big)\subseteq\mathfrak{p}.\]
For the reverse inclusion, consider $m\in \sigma^\vee\cap M$ with $m\notin\tau^\perp$. By assumption, $\langle m,v_\tau\rangle\in\mathbb{Z}$ and $\langle m,v_\tau\rangle\geq0$; moreover, since $m\notin\tau^\perp$, one has $\langle m,v_\tau\rangle\geq1$. For each $u=\sum_i\lambda_iv_i\in\sigma$ one has \[\langle m-v_\tau^\vee,u\rangle=\lambda_1\langle m-v_\tau^\vee,v_\tau\rangle+\sum_{i\geq2}\lambda_i\langle m,v_i\rangle\geq\lambda_1\left(\langle m,v_\tau\rangle-1\right)\geq0.\]As a result, $m-v_\tau^\vee\in\sigma^\vee\cap M$ and hence $\chi^m=\chi^{v_\tau^\vee}\cdot\chi^{m-v_\tau^\vee}\in\big(\chi^{v_\tau^\vee}\big)$, completing the proof.
\end{proof}

\begin{rem}
The last statement of the previous lemma is not true for a general strongly convex polyhedral rational cone $\sigma$ of maximal dimension in $N$. For example, if $\sigma$ has more than $n$ faces of dimension $1$, the divisor class group of $X_\sigma$, which is generated by the classes of the orbit closures associated to the rays, turns out to be nontrivial, as a consequence of \cite[Proposition at page 63]{Ful}. 
\end{rem}

For a nonzero Laurent polynomial $f\in K[M]$, the subset $V(f)$ of zeros of $f$ in $X_0$ is a closed subscheme of the dense open orbit. Its closure in $X_\Sigma$ is a closed subscheme of $X_\Sigma$, denoted by $\overline{V(f)}$. Taking into account multiplicities, one can consider the associated Weil divisor $\Big[\overline{V(f)}\Big]$. It is the zero cycle when $f$ is a monomial.

\begin{theorem}\label{divisor of Laurent polynomial}
Let $f$ be a nonzero Laurent polynomial, $\tilde{f}$ the rational function on $X_\Sigma$ arising from $f$. Then, \[\divisor(\tilde{f})=\Big[\overline{V(f)}\Big]+\sum_{\tau\in\Sigma^{(1)}}\Psi_{\NP(f)}(v_\tau)V(\tau),\]
where $\Psi_{\NP(f)}$ denotes the support function of the Newton polytope of $f$. In particular, $\Big[\overline{V(f)}\Big]$ is rationally equivalent to the cycle $-\sum_{\tau\in\Sigma^{(1)}}\Psi_{\NP(f)}(v_\tau)V(\tau)$ on $X_\Sigma$.
\end{theorem}
\begin{proof}
By \cite[formula at page 55]{Ful}, the irreducible components of $X_\Sigma\setminus X_0$ are exactly the orbit closures $V(\tau)$, with $\tau$ ranging in the set of $1$ dimensional cones of $\Sigma$. Since moreover the restriction of $\tilde{f}$ to $X_0$ is the regular function $f$, it follows from the classical theory of divisors that \[\divisor(\tilde{f})=\Big[\overline{V(f)}\Big]+\sum_\tau \nu_\tau(\tilde{f})V(\tau),\] where $\nu_\tau(\tilde{f})\in\mathbb{Z}$ is the order of vanishing of $\tilde{f}$ along $V(\tau)$. The statement of the theorem then follows from the fact that, for every $\tau\in\Sigma^{(1)}$, such an order equals $\Psi_{\NP(f)}(v_\tau)$.
\\This claim can be proved locally; fix a ray $\tau\in\Sigma^{(1)}$ and let $\sigma$ be any maximal dimensional cone of $\Sigma$ containing $\tau$. The fan being complete and consisting of smooth cones, such a $\sigma$ exists and the minimal integral vectors $v_1,\dots,v_n$ of its rays are a basis of $N$. Assume moreover that $v_1=v_\tau$ and, for simplicity, denote by $R:=K[\sigma^\vee\cap M]$ the ring of regular functions over $X_\sigma$. The order of vanishing of $\tilde{f}$ along $V(\tau)$ is computed as the valuation of $\tilde{f}$ determined by the valuation ring $R_{\mathfrak{p}}$, the localization of $R$ at the prime ideal $\mathfrak{p}$ corresponding to the subvariety $V(\tau)$ in $X_\sigma$. By \hyperref[lemma prime ideal of the orbit]{Lemma \ref*{lemma prime ideal of the orbit}}, the cone $\sigma$ being smooth and maximal dimensional, one has that $\mathfrak{p}=\big(\chi^{v_\tau^\vee}\big)$. The maximal ideal $\mathfrak{p}R_\mathfrak{p}$ of $R_\mathfrak{p}$ is hence the principal ideal generated by $\chi^{v_\tau^\vee}$.
\\Suppose first that $\tilde{f}=\sum_m c_m\chi^m$ lies in $R$, that is every $m$ appearing in $\tilde{f}$ belongs to $\sigma^\vee\cap M$. By definition of the valuation in $R_\mathfrak{p}$,
\begin{equation*}
\begin{split}
\nu_\tau(\tilde{f})&=\max\ \left\{l\in\mathbb{N}:\tilde{f}\in(\mathfrak{p}R_\mathfrak{p})^l\right\}=\max\ \left\{l\in\mathbb{N}:\tilde{f}\in\big(\chi^{lv_\tau^\vee}\big)\right\}\\&=\max\ \left\{l\in\mathbb{N}:\chi^{m-lv_\tau^\vee}\in R_\mathfrak{p}\text{ for all }m\text{ with } c_m\neq0\right\}.
\end{split}
\end{equation*}
The condition $\chi^{m-lv_\tau^\vee}\in R_\mathfrak{p}$ is equivalent to the fact that $\langle m,v_\tau\rangle\geq l$. Indeed, if the first is true, then \[\langle m,v_\tau\rangle-l=\langle m-lv_\tau^\vee,v_\tau\rangle\geq0.\] Conversely, for each $u=\sum_i\lambda_iv_i\in\sigma$ one has\[\langle m-lv_\tau^\vee,u\rangle=\lambda_1(\langle m,v_\tau\rangle-l)+\sum_{i\geq2}\lambda_i\langle m,v_i\rangle\geq0,\]and so $m-lv_\tau^\vee\in\sigma^\vee\cap M$. As a consequence,
\begin{equation*}
\begin{split}
\nu_\tau(\tilde{f})&=\max\ \left\{l\in\mathbb{N}:\langle m,v_\tau\rangle\geq l\text{ for all }m\text{ with } c_m\neq0\right\}\\&=\min\ \{\langle m,v_\tau\rangle: m\text{ with }c_m\neq0\}=\Psi_{\NP(f)}(v_\tau).
\end{split}
\end{equation*}
For a general $\tilde{f}=\sum_m c_m\chi^m$, the fact that $\sigma^\vee$ has dimension $n$ in $M_\mathbb{R}$ ($\sigma$ is indeed strongly convex) assures that there exists a big enough vector $m_0\in\sigma^\vee\cap M$ for which $m+m_0\in\sigma^\vee\cap M$ for each $m$ such that $c_m\neq0$. Hence\[\tilde{f}=\frac{\sum_mc_m\chi^{m+m_0}}{\chi^{m_0}},\]with both the numerator and the denominator belonging to $R$. Applying the result for such elements one deduces\[\nu_\tau(\tilde{f})=\nu_\tau\left(\sum_mc_m\chi^{m+m_0}\right)-\nu_\tau(\chi^{m_0})=\Psi_{\NP(f)+m_0}(v_\tau)-\langle m_0,v_\tau\rangle=\Psi_{\NP(f)}(v_\tau),\]concluding the proof.
\end{proof}

\section{Local and global heights of hypersurfaces}\label{section hypersurface}

Fix for the whole section an adelic field $\left(K,(|\cdot|_v,n_v)_{v\in\mathfrak{M}}\right)$ satisfying the product formula. Let $X_\Sigma$ be a proper toric variety over $K$, of dimension $n$, with torus $\mathbb{T}=\spec K[M]$ and dense open orbit $X_0$. For an effective cycle $Z$ on $X_\Sigma$ of pure codimension $1$, whose prime components intersect $X_0$, we present a series of results concerning its integrability and its local and global height with respect to suitable choices of metrized divisors on $X_\Sigma$.

\subsection{Degrees}

With the notations and assumptions given above, the effective cycle $Z$ can be written as \[Z=\sum_{i=1}^r\ell_i Y_i\] for prime divisors $Y_1,\dots,Y_r$ intersecting $X_0$. For every $i=1,\dots,r$, the closed irreducible subvariety of $X_0$ obtained as the intersection between $Y_i$ and $X_0$ is associated to a prime ideal of height one in $K[M]$, which is principal since $K[M]$ is a unique factorization domain; denote by $f_i$ an irreducible Laurent polynomial generating such an ideal. The Laurent polynomial $f=f_1^{\ell_1}\cdot\dots\cdot f_r^{\ell_r}$ is called a \emph{defining polynomial} for the cycle $Z$ and is uniquely defined up to multiplication by an invertible element of $K[M]$, that means by a monomial. Moreover,
\begin{equation}\label{cycle of defining polynomial}
\Big[\overline{V(f)}\Big]=Z,
\end{equation}
that is the cycle associated to the closure of the subscheme $V(f)$ in $X_\Sigma$ agrees with $Z$. Let \[\Psi_f:=\Psi_{\NP(f)}\] be the support function, in the sense of \hyperref[indicator and support function]{Example \ref*{indicator and support function}}, of the Newton polytope $\NP(f)$ of $f$; it is a piecewise linear function on $N_\mathbb{R}$. It is not necessarily a virtual support function on the fan $\Sigma$, but it can always be made such after a suitable refinement of the fan.

\begin{lemma}\label{a suitable resolution exists}
For any proper toric variety $X_\Sigma$ there exist a smooth projective toric variety $X_{\Sigma^\prime}$ with fan $\Sigma^\prime$ in $N_\mathbb{R}$ and a proper toric morphism $\pi:X_{\Sigma^\prime}\to X_\Sigma$ satisfying:
\begin{enumerate}
\item $\pi$ restricts to the identity on the dense open orbit of $X_{\Sigma^\prime}$ and $X_\Sigma$
\item $\Psi_f$ is a virtual support function on $\Sigma^\prime$.
\end{enumerate}
\end{lemma}
\begin{proof}
One can always refine the complete fan $\Sigma$ to a fan $\Sigma^\prime$ in such a way that $\Psi_f$ is a virtual support function on $\Sigma^\prime$. After possibly refining again, one can suppose that $\Sigma^\prime$ is the fan of a projective toric variety (because of the toric Chow lemma, see \cite[Theorem 6.1.18]{CLS}) and that each of its cones is generated by a part of a basis of $N$ (see \cite[\S2.6]{Ful}). The associated toric variety $X_{\Sigma^\prime}$ is smooth, projective and it satisfies (2). Finally, since $\Sigma^\prime$ is a refinement of $\Sigma$, the toric morphism $\pi$ given by \cite[Theorem 3.3.4]{CLS} is proper and restricts to the identity on the dense open orbit of $X_{\Sigma^\prime}$.
\end{proof}

The previous lemma, together with the fact that intersection theoretical properties are stable under birational transformations, allows to compute the degree of a cycle of codimension $1$ in a toric variety with respect to a family of toric divisors generated by global sections.

\begin{proposition}\label{degree of an hypersurface}
Let $D_{\Psi_1},\dots,D_{\Psi_{n-1}}$ be toric divisors on $X_\Sigma$ generated by global sections, $Z$ an effective cycle on $X_\Sigma$ of pure codimension $1$ and prime components intersecting $X_0$, with defining polynomial $f$. Then \[\deg_{D_{\Psi_1},\dots,D_{\Psi_{n-1}}}(Z)=\MV_M(\Delta_{\Psi_1},\dots,\Delta_{\Psi_{n-1}},\NP(f)),\]where $\MV_M$ denotes the mixed volume function associated to the measure $\vol_M$ (see \hyperref[lattice volume]{Remark \ref*{lattice volume}}) and $\Delta_{\Psi_i}$ the polytope associated to the toric divisor $D_{\Psi_i}$, for each $i=1,\dots,n-1$.
\end{proposition}
\begin{proof}
Consider the smooth projective toric variety $X_{\Sigma^\prime}$ and the proper toric morphism $\pi:X_{\Sigma^\prime}\to X_\Sigma$ given by \hyperref[a suitable resolution exists]{Lemma \ref*{a suitable resolution exists}}. Since the support function $\Psi_f$ is a virtual support function on $\Sigma^\prime$, one can consider the corresponding toric divisor $D_f$ on $X_{\Sigma^\prime}$ and the associated distinguished rational section $s_f$ of $\mathscr{O}(D_f)$. The product $\tilde{f}s_f$ is a rational section of $\mathscr{O}(D_f)$, with associated Weil divisor satisfying
\begin{equation*}
\pi_*\divisor\big(\tilde{f} s_f\big)=\pi_*(\divisor(\tilde{f})+\divisor(s_f))=Z
\end{equation*}
by \hyperref[divisor of Laurent polynomial]{Theorem \ref*{divisor of Laurent polynomial}}, \eqref{divisor of toric section}, \eqref{cycle of defining polynomial} and the definition of $\pi$. The projection formula in \cite[Proposition 2.3 (c)]{Ful2} and B\'ezout's theorem yield
\[\deg_{D_{\Psi_1},\dots,D_{\Psi_{n-1}}}(Z)=\deg_{\pi^*D_{\Psi_1},\dots,\pi^*D_{\Psi_n-1},D_f}(X_{\Sigma^\prime}).\]
The function $\Psi_f$ being concave, $D_f$ is generated by global sections. Moreover, the virtual support functions associated to the toric divisor $\pi^*D_{\Psi_i}$ on $X_{\Sigma^\prime}$ agrees with $\Psi_i$, for every $i=1,\dots,n-1$. The combinatorial description in \cite[Proposition 2.10]{O} of the degree of a toric variety with respect to toric divisor generated by global sections concludes then the proof.
\end{proof}

\begin{rem}\label{remark for degree of toric divisor}
By \cite[formula at page 55]{Ful}, the irreducible components of $X_\Sigma\setminus X_0$ are the orbit closures $V(\tau)$, with $\tau$ ranging in the set of $1$ dimensional cones of $\Sigma$. It follows that if $Z$ is a prime divisor of $X_\Sigma$ not intersecting $X_0$, it coincides with $V(\tau)$ for some $\tau\in\Sigma^{(1)}$. In such a case, the degree of $Z$ with respect to a collection $D_{\Psi_1},\dots,D_{\Psi_{n-1}}$ of toric divisors on $X_\Sigma$ generated by global sections is given by
\[\deg_{D_{\Psi_1},\dots,D_{\Psi_{n-1}}}(V(\tau))=\MV_{M(v_\tau)}\big(\Delta_{\Psi_1}^{v_\tau},\dots,\Delta_{\Psi_{n-1}}^{v_\tau}\big),\]
where $v_\tau$ is the minimal nonzero integral vector of $\tau$, see \cite[formul{\ae} (3.4.1) and (3.4.4)]{BPS}.
\end{rem}

\begin{rem}\label{can reduce to the case of a toric variety compatible with the polynomial}
The reduction to the case of a smooth projective toric variety employed in the proof of \hyperref[degree of an hypersurface]{Proposition \ref*{degree of an hypersurface}} equally works when computing the local height of the cycle $Z$ with respect to a family of $v$-adic semipositive toric metrized divisors $\overline{D}_{0,v},\dots,\overline{D}_{n-1,v}$. Indeed, let $f$ be a defining polynomial for $Z$, $X_{\Sigma^\prime}$ and $\pi$ as in the statement of \hyperref[a suitable resolution exists]{Lemma \ref*{a suitable resolution exists}}. For every family of rational sections $s_0,\dots,s_{n-1}$ of $\mathscr{O}(D_0),\dots,\mathscr{O}(D_{n-1})$ respectively for which the following local heights are well-defined, the local arithmetic projection formula in \cite[Theorem 1.4.17 (2)]{BPS} asserts that\[h_{\overline{D}_{0,v},\dots,\overline{D}_{n-1,v}}(Z;s_0,\dots,s_{n-1})=h_{\pi^*\overline{D}_{0,v},\dots,\pi^*\overline{D}_{n-1,v}}(Z^\prime;\pi^*s_0,\dots,\pi^*s_{n-1}),\] where $Z^\prime$ is the cycle in $X_{\Sigma^\prime}$ associated to the subscheme obtained as the closure of $V(f)$ and has hence $f$ as a defining polynomial. Because of \cite[Proposition 4.3.19]{BPS}, the pull-back of $\overline{D}_{i,v}$ via $\pi$ is a $v$-adic semipositive toric metrized divisor on $X_{\Sigma^\prime}$ whose metric function coincides with the one of $\overline{D}_{i,v}$, for every $i=0,\dots,n-1$ and $v\in\mathfrak{M}$. It follows that any combinatorial formula for the local height of $Z^\prime$ in $X_{\Sigma^\prime}$ with respect to $\pi^*\overline{D}_{0,v},\dots,\pi^*\overline{D}_{n-1,v}$ only involving the defining polynomial of $Z$ and the metric functions of the metrized divisors equally holds for the local height of $Z$ in $X_\Sigma$ with respect to $\overline{D}_{0,v},\dots,\overline{D}_{n-1,v}$. Similarly, the reduction step can be adopted when dealing with the integrability and the global height of $Z$, because of \hyperref[heights and proper morphisms]{Proposition \ref*{heights and proper morphisms}}.
\end{rem}

\subsection{Local heights}\label{section about local heights}

Let $f$ be a defining polynomial for the cycle $Z$ and, as in the previous subsection, denote by $\Psi_f$ the support function of its Newton polytope. Under the assumption that $\Psi_f$ is a virtual support function on the fan of $X_\Sigma$, it determines a toric divisor $D_f$ on $X_\Sigma$ together with a distinguished rational section $s_f$ of $\mathscr{O}(D_f)$, as in \hyperref[section height on toric varieties]{subsection \ref*{section height on toric varieties}}.

\begin{defn}\label{v-adic Ronkin metric}
In the notation above, and for a place $v\in\mathfrak{M}$, the \emph{$v$-adic Ronkin metric} on $D_f$ is the $v$-adic semipositive toric metric on $\mathscr{O}(D_f)_v^{\an}$ corresponding to the $v$-adic Ronkin function $\rho_{f,v}$ via \hyperref[characterization of toric metrics]{Theorem \ref*{characterization of toric metrics}}.
\end{defn}

The previous definition makes sense since, for every $v\in\mathfrak{M}$, the $v$-adic Ronkin function of $f$ is concave on $N_\mathbb{R}$ and has bounded difference from $\Psi_f$ because of \hyperref[properties of Ronkin functions]{Proposition \ref*{properties of Ronkin functions}}. If not otherwise specified, $\overline{D}_{f,v}$ will denote the divisor $D_f$ equipped with the $v$-adic Ronkin metric $\|\cdot\|_{f,v}$ defined above. By definition,
\begin{equation}\label{Ronkin metric definition}
\log\|s_f\|_{f,v}=\rho_{f,v}\circ\trop_v
\end{equation}
on $X_{0,v}^{\an}$. To lighten the notation, we will drop the subscript $v$ whenever the choice of the place is clear from the context.

\begin{proposition}
Let $f$ and $g$ be two nonzero Laurent polynomials and assume that $\Psi_f$ and $\Psi_g$ are virtual support functions on the fan of $X_\Sigma$. Then, $\overline{D}_f+\overline{D}_g=\overline{D}_{f\cdot g}$.
\end{proposition}
\begin{proof}
The equality $\NP(f\cdot g)=\NP(f)+\NP(g)$ implies that $\Psi_{f\cdot g}=\Psi_f+\Psi_g$. In particular, $\Psi_{f\cdot g}$ is a virtual support function on the fan $\Sigma$ and then defines a toric divisor $D_{f\cdot g}$ on $X_\Sigma$ which satisfies $D_{f\cdot g}=D_f+D_g$ because of \cite[\S3.4]{Ful}. The statement follows now from \cite[Proposition 4.3.14 (1)]{BPS} and \hyperref[Ronkin function of a product]{Proposition \ref*{Ronkin function of a product}}.
\end{proof}

The key property of the Ronkin metric is given in the following proposition.

\begin{proposition}\label{local height of hypersurfaces}
Let $X_\Sigma$ be a smooth projective toric variety, $Z$ an effective cycle on $X_\Sigma$ of pure codimension $1$ and prime components intersecting $X_0$. Let $f$ be a defining polynomial for $Z$ and $\tilde{f}$ the associated rational function on $X_\Sigma$. Assume moreover that $\Psi_f$ is a virtual support function on the fan $\Sigma$. For a fixed place $v$ of $K$, let $\overline{D}_0,\dots,\overline{D}_{n-1}$ be toric divisors on $X_\Sigma$, equipped with $v$-adic semipositive toric metrics. Then
\begin{equation}\label{local height of hypersurfaces equation}
h_{\overline{D}_0,\dots,\overline{D}_{n-1}}(Z;s_0,\dots,s_{n-1})=h_{\overline{D}_0,\dots,\overline{D}_{n-1},\overline{D}_f}\big(X_\Sigma;s_0,\dots,s_{n-1},\tilde{f}s_f\big),
\end{equation}
for every choice of rational sections $s_0,\dots,s_{n-1}$ of $\mathscr{O}(D_0),\dots,\mathscr{O}(D_{n-1})$ respectively with $\divisor(s_0),\dots,\divisor(s_{n-1}),Z$ intersecting properly. 
\end{proposition}
\begin{proof}
The product $\tilde{f}s_f$ is a rational section of $\mathscr{O}(D_f)$ on $X_\Sigma$ with associated Weil divisor
\begin{equation*}
\divisor\big(\tilde{f} s_f\big)=\divisor(\tilde{f})+\divisor(s_f)=Z
\end{equation*}
by \hyperref[divisor of Laurent polynomial]{Theorem \ref*{divisor of Laurent polynomial}}, \eqref{divisor of toric section} and \eqref{cycle of defining polynomial}. Hence, the sections $s_0,\dots,s_{n-1},\tilde{f}s_f$ meet $X_\Sigma$ properly and the right hand side term in \eqref{local height of hypersurfaces equation} is well defined.
\\\hyperref[local height definition]{Definition \ref*{local height definition}} stating that
\begin{multline*}
h_{\overline{D}_0,\dots,\overline{D}_{n-1}}(Z;s_0,\dots,s_{n-1})=h_{\overline{D}_0,\dots,\overline{D}_{n-1},\overline{D}_f}\big(X_\Sigma;s_0,\dots,s_{n-1},\tilde{f}s_f\big)\\+\int_{X_\Sigma^{\an}}\log\|\tilde{f}s_f\|_f\ c_1(\overline{D}_0)\wedge\dots\wedge c_1(\overline{D}_{n-1}),
\end{multline*}
the proposition follows from the vanishing of the integral on the right hand side. Indeed, thanks to \hyperref[Chambert-Loir measure in the toric case]{Theorem \ref*{Chambert-Loir measure in the toric case}}, such an integral is supported on the analytification of the dense open orbit of $X_\Sigma$, where the rational function $\tilde{f}$ coincides with the regular function $f$. Together with the definition of the Ronkin metric in \eqref{Ronkin metric definition}, this yields
\begin{multline*}
\int_{X_\Sigma^{\an}}\log\|\tilde{f}s_f\|_f\ c_1(\overline{D}_0)\wedge\dots\wedge c_1(\overline{D}_{n-1})=\int_{X_0^{\an}}\log|f|\ c_1(\overline{D}_0)\wedge\dots\wedge c_1(\overline{D}_{n-1})\\+\int_{X_0^{\an}}(\rho_{f,v}\circ\trop_v)\ c_1(\overline{D}_0)\wedge\dots\wedge c_1(\overline{D}_{n-1}).
\end{multline*}
For every $i=0,\dots,n-1$, denote by $\psi_i$ the metric function of $\overline{D}_i$. The tropicalization map being continuous, the change of variables formula and \hyperref[Chambert-Loir measure in the toric case]{Theorem \ref*{Chambert-Loir measure in the toric case}} imply on the one hand that
\begin{equation*}
\int_{X_0^{\an}}(\rho_{f,v}\circ\trop_v)\ c_1(\overline{D}_0)\wedge\dots\wedge c_1(\overline{D}_{n-1})=\int_{N_\mathbb{R}}\rho_{f,v}\ d\MM_M(\psi_0,\dots,\psi_{n-1}).
\end{equation*}
On the other hand, \hyperref[Chambert-Loir measure in the toric case]{Theorem \ref*{Chambert-Loir measure in the toric case}}, together with the change of variables formula and Fubini's theorem, gives
\begin{multline*}
\int_{X_0^{\an}}\log|f|\ c_1(\overline{D}_0)\wedge\dots\wedge c_1(\overline{D}_{n-1})=\\\int_{N_\mathbb{R}}\bigg(\int_{\mathcal{B}_{\mathbb{C}_v}}(\log|f|\circ\iota_v)\ d\Haar_{\mathcal{B}_{\mathbb{C}_v}}\bigg)\ d\MM_M(\psi_0,\dots,\psi_{n-1}).
\end{multline*}
The definition of the maps $\iota_v$ and $\rho_{f,v}$ assures that the inner integral coincides with the opposite of the $v$-adic Ronkin function of $f$, concluding the proof.
\end{proof}

\subsection{Toric local heights}

Recall from \hyperref[canonical metric definition]{Definition \ref*{canonical metric definition}} that any toric divisor generated by global sections admits a distinguished $v$-adic semipositive toric metric, the canonical metric. This allows to define a local height with respect to toric divisors that is independent of the choice of the sections. Such a notion was introduced in \cite[\S5.1]{BPS} as a key step in the proof of the formula for the global height of a toric variety. 

\begin{defn}
For a place $v$ of $K$, let $\overline{D}_0,\dots,\overline{D}_d$ be toric divisors on $X_\Sigma$, endowed with $v$-adic semipositive toric metrics. Denote by $\overline{D}_0^{\can},\dots,\overline{D}_d^{\can}$ the same divisors equipped with their $v$-adic canonical metric. Let $Y$ be an irreducible $d$-dimensional subvariety of $X_\Sigma$ and $\varphi:Y^\prime\to Y$ a birational morphism, with $Y^\prime$ projective. The \emph{$v$-adic toric local height} of $Y$ with respect to $\overline{D}_0,\dots,\overline{D}_{d}$ is defined as
\[h^{\tor}_{\overline{D}_0,\dots,\overline{D}_d}(Y):=h_{\varphi^*\overline{D}_0,\dots,\varphi^*\overline{D}_d}(Y^\prime;s_0,\dots,s_d)-h_{\varphi^*\overline{D}_0^{\can},\dots,\varphi^*\overline{D}_d^{\can}}(Y^\prime;s_0,\dots,s_d),\]
where $s_i$ is a rational section of $\varphi^*\mathscr{O}(D_i)$, for every $i=0,\dots,d$ and $s_0,\dots,s_d$ meet $Y^\prime$ properly. The definition extends by linearity to any cycle of dimension $d$.
\end{defn}

The toric local height of a cycle does neither depend on the choice of the sections $s_0,\dots,s_d$, nor on the birational model $Y^\prime$ of $Y$ because of \cite[Theorem 1.4.17 (2) and (3)]{BPS}. Moreover, the definition is nonempty: Chow's lemma provides $Y$ with a projective birational model, while the moving lemma assures the existence of rational sections meeting $Y^\prime$ properly.
\\We prove here a formula for the toric local height of an effective cycle $Z$ on $X_\Sigma$ of pure codimension $1$ and prime components intersecting $X_0$.

\begin{theorem}\label{toric local height of hypersurfaces}
Let $X_\Sigma$ be a proper toric variety, $Z$ an effective cycle on $X_\Sigma$ of pure codimension $1$ and prime components intersecting $X_0$. For a place $v$ of $K$, let $\overline{D}_0,\dots,\overline{D}_{n-1}$ be toric divisors on $X_\Sigma$, equipped with $v$-adic semipositive toric metrics. Then
\[h^{\tor}_{\overline{D}_0,\dots,\overline{D}_{n-1}}(Z)=\MI_M\big(\vartheta_0,\dots,\vartheta_{n-1},\rho_{f,v}^\vee\big)+\deg_{D_0,\dots,D_{n-1}}(X_\Sigma)\cdot\rho_{f,v}(0),\]
where $f$ is a defining polynomial for $Z$ and $\vartheta_i$ is the roof function of $\overline{D}_i$, for $i=0,\dots,n-1$.
\end{theorem}
\begin{proof}
Because of \hyperref[can reduce to the case of a toric variety compatible with the polynomial]{Remark \ref*{can reduce to the case of a toric variety compatible with the polynomial}}, one can assume that $X_\Sigma$ is a smooth projective toric variety on whose fan $\Psi_f$ is a virtual support function. Thanks to the moving lemma, one can choose rational sections $s_0,\dots,s_{n-1}$ of $\mathscr{O}(D_0),\dots,\mathscr{O}(D_{n-1})$ respectively such that $\divisor(s_0),\dots,\divisor(s_{n-1}),Z$ intersect properly. \hyperref[local height of hypersurfaces]{Proposition \ref*{local height of hypersurfaces}} implies then that
\begin{multline*}
h^{\tor}_{\overline{D}_0,\dots,\overline{D}_{n-1}}(Z)=h_{\overline{D}_0,\dots,\overline{D}_{n-1},\overline{D}_f}\big(X_\Sigma;s_0,\dots,s_{n-1},\tilde{f}s_f\big)\\-h_{\overline{D}_0^{\can},\dots,\overline{D}_{n-1}^{\can},\overline{D}_f}\big(X_\Sigma;s_0,\dots,s_{n-1},\tilde{f}s_f\big).
\end{multline*}
By adding and subtracting the quantity
\[h_{\overline{D}_0^{\can},\dots,\overline{D}_{n-1}^{\can},\overline{D}_f^{\can}}\big(X_\Sigma;s_0,\dots,s_{n-1},\tilde{f}s_f\big)\]
on the right hand side, one obtains that
\[h^{\tor}_{\overline{D}_0,\dots,\overline{D}_{n-1}}(Z)=h^{\tor}_{\overline{D}_0,\dots,\overline{D}_{n-1},\overline{D}_f}(X_\Sigma)-h^{\tor}_{\overline{D}_0^{\can},\dots,\overline{D}_{n-1}^{\can},\overline{D}_f}(X_\Sigma).\]
Denote by $\Psi_i$ the virtual support function on $\Sigma$ associated to the toric divisor $D_i$, for every $i=0,\dots,n-1$. Thanks to \cite[Corollary 5.1.9]{BPS}, the previous equality yields
\[h^{\tor}_{\overline{D}_0,\dots,\overline{D}_{n-1}}(Z)=\MI_M\big(\vartheta_0,\dots,\vartheta_{n-1},\rho_{f,v}^\vee\big)-\MI_M\big(\Psi_0^\vee,\dots,\Psi_{n-1}^\vee,\rho_{f,v}^\vee\big).\]
Since they admit by hypothesis a semipositive toric metric, the toric divisors $D_0,\dots,\allowbreak D_{n-1}$ are generated by global sections. For every $i=0,\dots,n-1$, the function $\Psi_i$ is hence concave and conic and so it is the support function of the polytope $\Delta_i:=\stab(\Psi_i)\subseteq M_\mathbb{R}$. The statement of the theorem follows from a combination of \hyperref[indicator and support function]{Example \ref*{indicator and support function}}, \hyperref[mixed integral with indicator functions]{Corollary \ref*{mixed integral with indicator functions}} and the combinatorial expression for the degree of a toric variety with respect to toric divisors generated by their global sections, see for example \cite[Proposition 2.10]{O}.
\end{proof}

\subsection{Global heights}

To state the result concerning the global case, recall that for a defining polynomial $f$ for $Z$, the support function of its Newton polytope is denoted by $\Psi_f$; whenever it is linear on each cone of a complete fan $\Sigma$, it defines a toric divisor $D_f$ on the toric variety $X_\Sigma$, together with a distinguished rational section $s_f$ of $\mathscr{O}(D_f)$.

\begin{defn}\label{definition of Ronkin metric}
In the above assumptions, the \emph{Ronkin metric} on $D_f$ is the choice, for every place $v\in\mathfrak{M}$, of the $v$-adic Ronkin metric on $D_f$ defined in \hyperref[v-adic Ronkin metric]{Definition \ref*{v-adic Ronkin metric}}.
\end{defn}

Unless otherwise stated, $\overline{D}_f$ will denote the toric divisor $D_f$ equipped with its Ronkin metric. By definition, it is a semipositive toric metrized divisor.

\begin{lemma}\label{Ronkin metric is adelic}
The Ronkin metric on $D_f$ is adelic.
\end{lemma}
\begin{proof}
For a non-archimedean place $v\in\mathfrak{M}$, the function $\rho_{f,v}$ coincides with the tropicalization of the Laurent polynomial $f$, as claimed in \hyperref[remark about known version of Ronkin functions]{Remark \ref*{remark about known version of Ronkin functions}}. The fact that $f$ has finitely many nonzero coefficients and the second axiom in \hyperref[adelic family]{Definition \ref*{adelic family}} imply that $\rho_{f,v}=\Psi_f$ for all but finitely many non-archimedean places. The statement follows then from \hyperref[adelic fields have finitely many archimedean places]{Lemma \ref*{adelic fields have finitely many archimedean places}}.
\end{proof}

The definition of such a toric metrized divisor and the study of the local height of $Z$ in \hyperref[section about local heights]{section \ref*{section about local heights}} allow to answer the question of the integrability of $Z$ and to give a formula for its global height, implying \hyperref[main theorem introduction]{Theorem \ref*{main theorem introduction}} in the introduction.

\begin{theorem}\label{hypersurfaces are integrable and their global height}
Let $X_\Sigma$ be a proper toric variety, $Z$ an effective cycle on $X_\Sigma$ of pure codimension $1$ and prime components intersecting $X_0$. Let $\overline{D}_0,\dots,\overline{D}_{n-1}$ be toric divisors on $X_\Sigma$, equipped with adelic semipositive toric metrics. Then, $Z$ is integrable with respect to $\overline{D}_0,\dots,\overline{D}_{n-1}$ and its global height is given by
\[h_{\overline{D}_0,\dots,\overline{D}_{n-1}}(Z)=\sum_{v\in\mathfrak{M}}n_v\MI_M\big(\vartheta_{0,v},\dots,\vartheta_{n-1,v},\rho_{f,v}^\vee\big),\]
where $f$ is a defining polynomial for $Z$ and $\vartheta_{i,v}$ is the roof function of $\overline{D}_{i,v}$, for every $i=0,\dots,n-1$ and $v\in\mathfrak{M}$.
\end{theorem}
\begin{proof}
Because of \hyperref[can reduce to the case of a toric variety compatible with the polynomial]{Remark \ref*{can reduce to the case of a toric variety compatible with the polynomial}}, one can assume that $X_\Sigma$ is a smooth projective toric variety on whose fan $\Psi_f$ is a virtual support function. Let hence $s_0,\dots,s_{n-1}$ be rational sections of $\mathscr{O}(D_0),\dots,\mathscr{O}(D_{n-1})$ respectively such that $\divisor(s_0),\dots,\divisor(s_{n-1}),Z$ intersect properly. Because of \hyperref[local height of hypersurfaces]{Proposition \ref*{local height of hypersurfaces}}, the $v$-adic local height of $Z$ with respect to the above choice of sections is given by
\begin{equation}\label{key equality in the proof of global heights}
h_{\overline{D}_{0,v},\dots,\overline{D}_{n-1,v}}(Z;s_0,\dots,s_{n-1})=h_{\overline{D}_{0,v},\dots,\overline{D}_{n-1,v},\overline{D}_{f,v}}\big(X_\Sigma;s_0,\dots,s_{n-1},\tilde{f}s_f\big).
\end{equation}
Because of \hyperref[Ronkin metric is adelic]{Lemma \ref*{Ronkin metric is adelic}}, each member of the family $\overline{D}_0,\dots,\overline{D}_{n-1},\overline{D}_f$ is an adelic semipositive toric metrized divisor on $X_\Sigma$. As a consequence of the first assertion in \cite[Proposition 5.2.4]{BPS}, $X_\Sigma$ is integrable with respect to such a choice of metrized divisors and hence \cite[Proposition 1.5.8 (1)]{BPS} allows to conclude that \[h_{\overline{D}_{0,v},\dots,\overline{D}_{n-1,v},\overline{D}_{f,v}}\big(X_\Sigma;s_0,\dots,s_{n-1},\tilde{f}s_f\big)=0\] for all but finitely many places $v\in\mathfrak{M}$. Comparing with \eqref{key equality in the proof of global heights}, one deduces that $Z$ is integrable with respect to $\overline{D}_0,\dots,\overline{D}_{n-1}$.
\\From the same equality, the global height of $Z$ is seen to satisfy
\[h_{\overline{D}_0,\dots,\overline{D}_{n-1}}(Z)=h_{\overline{D}_0,\dots,\overline{D}_{n-1},\overline{D}_f}(X_\Sigma).\]The formula for the global height of $Z$ follows then from \hyperref[height of a toric variety]{Theorem \ref*{height of a toric variety}}.
\end{proof}

\begin{rem}
As in \hyperref[remark for degree of toric divisor]{Remark \ref*{remark for degree of toric divisor}}, if $Z$ is an irreducible hypersurface on $X_\Sigma$ not intersecting $X_0$ it coincides with $V(\tau)$ for a $1$ dimensional cone $\tau$ of the fan $\Sigma$. In such a case, $Z$ is integrable with respect to a family of adelic semipositive toric metrized divisors $\overline{D}_0,\dots,\overline{D}_{n-1}$ on $X_\Sigma$ and its global height is given by
\[h_{\overline{D}_0,\dots,\overline{D}_{n-1}}(V(\tau))=\sum_{v\in\mathfrak{M}}n_v\MI_{M(v_\tau)}\big(\vartheta_{0,v}|_{\Delta_0^{v_\tau}},\dots,\vartheta_{n-1,v}|_{\Delta_{n-1}^{v_\tau}}\big),\]
see \cite[Proposition 5.1.11 and Proposition 5.2.4]{BPS}. In the previous formula, $\Delta_i$ is the polytope associated to the divisor $D_i$ and $\vartheta_{i,v}$ is the roof function of $\overline{D}_{i,v}$ for every $i=0,\dots,n-1$ and $v\in\mathfrak{M}$, while $v_\tau$ is the minimal nonzero integral vector of $\tau$.
\end{rem}

\begin{rem}
Local and global heights of cycles are symmetric and multilinear with respect to sums of semipositive metrized divisors, provided that all terms are defined. The formulas obtained for $1$-codimensional cycles in toric varieties are consistent with these properties, the sum of semipositive toric metrized divisors corresponding to the sup-convolution of the associated roof functions, as shown in \cite[4.3.14 (1)]{BPS}.
\end{rem}

\section{Examples}\label{section about examples}

For a fixed adelic field $\left(K,(|\cdot|_v,n_v)_{v\in\mathfrak{M}}\right)$ satisfying the product formula and a proper toric variety $X_\Sigma$ over $K$, of dimension $n$, with torus $\mathbb{T}=\spec K[M]$ and dense open orbit $X_0$, we apply in this section the formula in \hyperref[hypersurfaces are integrable and their global height]{Theorem \ref*{hypersurfaces are integrable and their global height}} to four particular cases. In the first one, we focus on specific hypersurfaces of $X_\Sigma$, while in the following three we made relevant choices of the metrized divisors.

\subsection{Binomial hypersurfaces}

For a primitive vector $m$ in $M$ one can consider the Laurent binomial $f=\chi^m-1$; it is irreducible in $K[M]$ as can be verified by considering its Newton polytope. Hence the closure $Z$ in $X_\Sigma$ of the subvariety $V(f)$ of the torus $\spec K[M]$ is an irreducible hypersurface of $X_\Sigma$ with defining polynomial $f$.
\\Let $\overline{D}_0,\dots,\overline{D}_{n-1}$ be toric divisors on $X_\Sigma$, equipped with adelic semipositive toric metrics, with $\vartheta_{i,v}$ the roof function of $\overline{D}_{i,v}$, for every $i=0,\dots,n-1$ and $v\in\mathfrak{M}$. By \hyperref[example Ronkin of a binomial]{Example \ref*{example Ronkin of a binomial}}, $\rho_{f,v}$ coincides for every $v\in\mathfrak{M}$ with the support function of the segment $\overline{0m}$ in $M_\mathbb{R}$. The formula in \hyperref[hypersurfaces are integrable and their global height]{Theorem \ref*{hypersurfaces are integrable and their global height}} implies then that $Z$ is integrable with respect to $\overline{D}_0,\dots,\overline{D}_{n-1}$ and that its global height is given by
\[h_{\overline{D}_0,\dots,\overline{D}_{n-1}}(Z)=\sum_{v\in\mathfrak{M}}n_v\MI_M\big(\vartheta_{0,v},\dots,\vartheta_{n-1,v},\iota_{\overline{0m}}\big),\] because of \hyperref[indicator and support function]{Example \ref*{indicator and support function}}. Considering, as at the end of \hyperref[subsection about mixed integrals]{subsection \ref*{subsection about mixed integrals}}, the quotient lattice $P:=M/\mathbb{Z}m$ and the associated projection $\pi:M\to P$, \hyperref[mixed integral with indicator function of a segment]{Proposition \ref*{mixed integral with indicator function of a segment}} allows to deduce
\begin{equation}\label{height of toric subvariety}
h_{\overline{D}_0,\dots,\overline{D}_{n-1}}(Z)=\sum_{v\in\mathfrak{M}}n_v\MI_P\big(\pi_*\vartheta_{0,v},\dots,\pi_*\vartheta_{n-1,v}\big),
\end{equation}
with $\pi_*\vartheta_{i,v}$ denoting the direct image of $\vartheta_{i,v}$ by $\pi$ for every $i=0,\dots,n-1$ and $v\in\mathfrak{M}$, see \eqref{definition of direct image of concave functions}.

\begin{rem}
Let $Q$ be the dual lattice of $P=M/\mathbb{Z}m$. The projection $\pi:M\to P$ induces an injective dual map $Q\to N$, with image $m^\perp\cap N$. By identifying $Q$ with such an image, which is a saturated sublattice of $N$, one can consider the restriction of the fan $\Sigma$ to $Q_\mathbb{R}$; its corresponding toric variety $X_{\Sigma_Q}$ is proper and has torus $\spec K[P]$. It also comes with a toric morphism $\varphi:X_{\Sigma_Q}\to X_\Sigma$, whose restriction to the dense open orbit coincides with the closed immersion of split tori $\spec K[P]\to\spec K[M]$ given by the surjection $\pi:M\to P$, see \cite[pages 81-83]{BPS}. Finally, the push-forward of the cycle $X_{\Sigma_Q}$ by $\varphi$ is the cycle $Z$ associated to the hypersurface defined by $\chi^m-1$. Indeed, the image of $\varphi$ coincides by properness with the closure in $X_\Sigma$ of the image of $\spec K[P]\to\spec K[M]$, which is an irreducible $(n-1)$-dimensional subscheme of $\spec K[M]$ contained in $V(\chi^m-1)$ as $\chi^{\pi(m)}-1=0$.
\\Hence, equality \eqref{height of toric subvariety} can also be obtained from \hyperref[height of a toric variety]{Theorem \ref*{height of a toric variety}}, the arithmetic projection formula and the fact that $\pi_*\vartheta_{i,v}$ is the roof function of the pull-back of $\overline{D}_{i,v}$ via $\pi$, for every $i=0,\dots,n-1$ and $v\in\mathfrak{M}$, because of \cite[Proposition 4.3.19 and Proposition 2.3.8 (3)]{BPS}.
\end{rem}

\subsection{The canonical height}

A toric divisor $D$ on $X_\Sigma$ generated by global sections admits by \hyperref[canonical metric definition]{Definition \ref*{canonical metric definition}} a distinguished semipositive toric metric at any place. The metrized divisor obtained by the choice of such a family of $v$-adic canonical metrics is denoted by $\overline{D}^{\can}$; it is an adelic semipositive toric metrized divisor.
\\For a cycle $Z$ of dimension $d$ in $X_\Sigma$, the \emph{canonical global height} of $Z$ with respect to a family $D_0,\dots,D_d$ of toric divisors on $X_\Sigma$ generated by global sections is defined to be its global height with respect to $\overline{D}^{\can}_0,\dots,\overline{D}^{\can}_d$ and it is also denoted by $h^{\can}_{D_0,\dots,D_d}(Z)$. The machinery developed in the previous sections allows to express the canonical global height of an effective cycle on $X_\Sigma$ of pure codimension $1$ via convex geometry.

\begin{proposition}\label{canonical height of hypersurfaces}
Let $Z$ be an effective cycle on $X_\Sigma$ of pure codimension $1$ and prime components intersecting $X_0$ and $D_0,\dots,D_{n-1}$ a family of toric divisors on $X_\Sigma$ generated by global sections. The canonical global height of $Z$ with respect to $D_0,\dots,D_{n-1}$ is given by \[h^{\can}_{D_0,\dots,D_{n-1}}(Z)=-\deg_{D_0,\dots,D_{n-1}}(X_\Sigma)\cdot\sum_{v\in\mathfrak{M}}n_v\rho_{f,v}(0)\] for any choice of a defining polynomial $f$ for $Z$.
\end{proposition}
\begin{proof}
Denoting for any $i=0,\dots,n-1$ by $\Psi_i$ the function associated to $D_i$, the property of being globally generated implies that $\Psi_i$ is the support function of the lattice polytope $\Delta_i:=\stab(\Psi_i)\subseteq M_\mathbb{R}$. The roof function of $\overline{D}^{\can}_{i,v}$ is hence $\iota_{\Delta_i}$ for every $i=0,\dots,n-1$ and for every $v\in\mathfrak{M}$, because of \hyperref[indicator and support function]{Example \ref*{indicator and support function}}. It follows from \hyperref[hypersurfaces are integrable and their global height]{Theorem \ref*{hypersurfaces are integrable and their global height}} and \hyperref[mixed integral with indicator functions]{Corollary \ref*{mixed integral with indicator functions}} that
\[h^{\can}_{D_0,\dots,D_{n-1}}(Z)=-\MV_M(\Delta_0,\dots,\Delta_{n-1})\cdot\sum_{v\in\mathfrak{M}} n_v\rho_{f,v}(0),\]
with $f$ any defining polynomial for $Z$. To conclude, recall that the degree of $X_\Sigma$ with respect to $D_0,\dots,D_{n-1}$ is given by the mixed volume of the associated polytopes, as proved in \cite[Proposition 2.10]{O}.
\end{proof}

The case of the base field $\mathbb{Q}$ with the adelic structure described in \hyperref[examples of adelic fields]{Example \ref*{examples of adelic fields}} is particularly interesting for arithmetic purposes. For a Laurent polynomial $f$ in $n$ variables and complex coefficients, one defines its (\emph{logarithmic}) \emph{Mahler measure} to be \[m(f):=\frac{1}{(2\pi)^n}\int_{\theta_1,\dots,\theta_n\in[0,2\pi]}\log\left|f\big(e^{i\theta_1},\dots,e^{i\theta_n}\big)\right|\ d\theta_1\dots d\theta_n.\]Such a quantity is notoriously difficult to compute and is sometimes related to special values of $L$-functions, see \cite{Smyth}, \cite{Dening}, \cite{Boyd} and \cite{Lalin}.
\\In \cite[Proposition 7.2.1]{Mail}, Maillot expressed the canonical height of a hypersurface in a toric variety over $\mathbb{Q}$ in terms of the Mahler measure of the associated section. While its proof relies on the study of the arithmetic Chow ring of the ambient toric variety, we here deduce his result from \hyperref[canonical height of hypersurfaces]{Proposition \ref*{canonical height of hypersurfaces}}.

\begin{corollary}[Maillot]\label{canonical height over the rationals}
In the hypotheses and notations  of \hyperref[canonical height of hypersurfaces]{Proposition \ref*{canonical height of hypersurfaces}}, assume moreover that the base adelic field is $\mathbb{Q}$ with its usual adelic structure. Let $f$ be a defining polynomial for $Z$ having as coefficients integer numbers with greatest common divisor $1$. Then,\[h^{\can}_{D_0,\dots,D_{n-1}}(Z)=\deg_{D_0,\dots,D_{n-1}}(X_\Sigma)\cdot m(f).\]
\end{corollary}
\begin{proof}
Let $f$ be a defining polynomial for $Z$ satisfying the assumptions. Because of \hyperref[remark about known version of Ronkin functions]{Remark \ref*{remark about known version of Ronkin functions}}, for every non-archimedean place $v$ of $\mathbb{Q}$\[\rho_{f,v}(0)=f^{\trop}(0)=0.\] At the unique archimedean place $v$ of $\mathbb{Q}$ one has by definition that $\rho_{f,v}(0)=-m(f)$. The statement follows then directly from \hyperref[canonical height of hypersurfaces]{Proposition \ref*{canonical height of hypersurfaces}}.
\end{proof}

\subsection{The $\rho$-height}

The strategy adopted in the \hyperref[section hypersurface]{previous section} to prove the main results of the paper suggests the introduction of a distinguished height function. Let $Z$ be an effective cycle on $X_\Sigma$ of pure codimension $1$ and prime components intersecting $X_0$ and assume that the support function of the Newton polytope of a defining polynomial for $Z$ is a virtual support function on the fan $\Sigma$. By \hyperref[a suitable resolution exists]{Lemma \ref*{a suitable resolution exists}}, this is always the case up to a birational toric transformation. In this setting, the choice of a defining polynomial $f$ for $Z$ determines a toric divisor $D_f$ on $X_\Sigma$ and a distinguished toric metric on it, the Ronkin metric, as introduced in \hyperref[definition of Ronkin metric]{Definition \ref*{definition of Ronkin metric}}. The so-obtained metrized divisor, which is denoted by $\overline{D}_f$, is an adelic semipositive toric metrized divisor by \hyperref[Ronkin metric is adelic]{Lemma \ref*{Ronkin metric is adelic}}.

\begin{defn}
In the above hypotheses and notations, the \emph{$\rho$-height} of $Z$, denoted by $h_{\rho}(Z)$, is defined as its global height with respect to $\overline{D}_f,\dots,\overline{D}_f$, for a choice of a defining polynomial $f$ for $Z$.
\end{defn}

As shown below, the $\rho$-height of $Z$ is independent of the choice of the defining polynomial $f$. Even if it is not clear whether such a height has a significant geometrical interpretation or arithmetical application, it is straightforward to give a combinatorial formula for it.

\begin{proposition}\label{formula for the rho height}
In the above hypotheses and notations, the $\rho$-height of $Z$ is given by \[h_\rho(Z)=(n+1)!\sum_{v\in\mathfrak{M}}n_v\int_{\NP(f)}\rho_{f,v}^\vee\ d\vol_M,\]where $f$ is a defining polynomial for $Z$ and $\NP(f)$ is its Newton polytope.
\end{proposition}
\begin{proof}
The statement follows trivially from \hyperref[hypersurfaces are integrable and their global height]{Theorem \ref*{hypersurfaces are integrable and their global height}}, \hyperref[properties of Ronkin functions]{Proposition \ref*{properties of Ronkin functions}} (3) and the properties of mixed integrals.
\end{proof}

\begin{rem}
The equality in \hyperref[formula for the rho height]{Proposition \ref*{formula for the rho height}} shows that the $\rho$-height of $Z$ does not depend on the choice of a defining polynomial for it. Indeed, if $f^\prime$ is another such polynomial, it must satisfy $f^\prime=c\cdot \chi^m\cdot f$ for some nonzero monomial $c\cdot \chi^m\in K[M]$. For every $v\in\mathfrak{M}$, one has then that \[\rho_{f^\prime,v}=-\log|c|_v+m+\rho_{f,v}\] by \hyperref[Ronkin function of a product]{Proposition \ref*{Ronkin function of a product}} and \hyperref[example Ronkin of a monomial]{Example \ref*{example Ronkin of a monomial}}. The stated independence follows hence from the relation \[\rho_{f^\prime,v}^\vee=\tau_m\rho_{f,v}^\vee+\log|c|_v\] obtained using \cite[Proposition 2.3.3]{BPS} and from the product formula on $K$.
\end{rem}

It is significant to stress that the formula in \hyperref[formula for the rho height]{Proposition \ref*{formula for the rho height}}, though compact, is difficult to evaluate because of the complexity of the archimedean Ronkin function.

\subsection{The Fubini-Study height}

As a last example, consider the ambient toric variety $X_\Sigma$ to be the $n$-dimensional projective space over $K$. Denote by $D_\infty$ the toric divisor on $\mathbb{P}^n_K$ whose associated Weil divisor is the hyperplane at infinity; the corresponding sheaf is the universal line bundle $\mathscr{O}(1)$ on $\mathbb{P}^n_K$. If not otherwise specified, the notation $\overline{D}_\infty$ will refer to $D_\infty$ equipped with the Fubini-Study metric at archimedean places, see \cite[Example 1.1.2]{BPS}, and the canonical one at non-archimedean places, in the sense of \hyperref[canonical metric definition]{Definition \ref*{canonical metric definition}}. It turns out that $\overline{D}_\infty$ is an adelic semipositive toric metrized divisor. Thanks to \hyperref[hypersurfaces are integrable and their global height]{Theorem \ref*{hypersurfaces are integrable and their global height}}, any effective cycle $Z$ on $\mathbb{P}^n_K$ of pure codimension $1$ is then integrable with respect to $\overline{D}_\infty,\dots,\overline{D}_\infty$ and the corresponding global height \[h_{\FS}(Z):=h_{\overline{D}_\infty,\dots,\overline{D}_\infty}(Z)\]is called the \emph{Fubini-Study height} of $Z$.

\begin{rem}
The Fubini-Study height defined here coincides with the one introduced in \cite{Fal2} and studied in \cite{PhilIII}. Examples of the computation of such height for projective hypersurfaces can be found in \cite{CM}.
\end{rem}

Specializing \hyperref[hypersurfaces are integrable and their global height]{Theorem \ref*{hypersurfaces are integrable and their global height}}, one can write the Fubini-Study height of a projective hypersurface in terms of convex geometry. To do so, denote by $\mathfrak{M}_\infty$ the collection of archimedean places of $K$, which is a finite set by \hyperref[adelic fields have finitely many archimedean places]{Lemma \ref*{adelic fields have finitely many archimedean places}}. After fixing an isomorphism $M\simeq\mathbb{Z}^n$, consider the standard simplex \[\Delta^n:=\big\{(x_1,\dots,x_n): x_1+\dots+x_n\leq1,x_i\geq0\text{ for all }i=1,\dots,n\big\}\] in $M_\mathbb{R}\simeq\mathbb{R}^n$ and, agreeing that $x_0:=1-\sum_{i=1}^nx_i$, set the function $\vartheta_{\FS}:\Delta^n\to\mathbb{R}$ to be \[\vartheta_{\FS}(x):=-\frac{1}{2}\sum_{i=0}^nx_i\log x_i,\]which is defined on the boundary of $\Delta^n$ by continuity.

\begin{proposition}\label{Fubini-Study height of a hypersurface}
Let $Z$ be an effective cycle on $\mathbb{P}^n_K$ of pure codimension $1$ and prime components intersecting $X_0$. The Fubini-Study height of $Z$ is given by\[h_{\FS}(Z)=\sum_{v\in\mathfrak{M}_\infty}n_v\MI_M\big(\vartheta_{\FS},\dots,\vartheta_{\FS},\rho_{f,v}^\vee\big)-\sum_{v\in\mathfrak{M}\setminus\mathfrak{M}_\infty}n_v\rho_{f,v}(0),\]where $f$ is a defining polynomial for $Z$.
\end{proposition}
\begin{proof}
The roof functions of the metrized divisor $\overline{D}_\infty$ are given by the function $\vartheta_{\FS}$ at archimedean places, as remarked in \cite[Example 2.4.3 and Example 4.3.9 (2)]{BPS} and by the indicator function of $\Delta^n$ at non-archimedean places, by \cite[Example 4.3.9 (1)]{BPS} and \hyperref[indicator and support function]{Example \ref*{indicator and support function}}. The statement follows then from \hyperref[hypersurfaces are integrable and their global height]{Theorem \ref*{hypersurfaces are integrable and their global height}} and \hyperref[mixed integral with indicator functions]{Corollary \ref*{mixed integral with indicator functions}}, together with the fact that $\MV_{M}(\Delta^n,\dots,\Delta^n)=1$ because of the conventions introduced in \hyperref[subsection about real Monge-Ampere measures]{subsection \ref*{subsection about real Monge-Ampere measures}} and \hyperref[lattice volume]{Remark \ref*{lattice volume}}.
\end{proof}

The non-archimedean contributions to the Fubini-Study height are easily computable, since for every Laurent polynomial $f$ with set of coefficients $\Gamma(f)$,
\[-\rho_{f,v}(0)=\log\max_{c\in\Gamma(f)}|c|_v\]
if $v\in\mathfrak{M}\setminus\mathfrak{M}_\infty$. From this equality one easily obtains the following special case.

\begin{corollary}\label{FS height over the rational}
Assume the base adelic field to be $\mathbb{Q}$ with its usual adelic structure. The Fubini-Study height of an effective cycle $Z$ on $\mathbb{P}^n_\mathbb{Q}$ of pure codimension $1$ and prime components intersecting $X_0$ is given by \[h_{\FS}(Z)=\MI_M\big(\vartheta_{\FS},\dots,\vartheta_{\FS},\rho_{f,\infty}^\vee\big),\]where $f$ is a defining polynomial for $Z$ whose coefficients are integer numbers with greatest common divisor $1$.
\end{corollary}

Because of the presence of an archimedean Ronkin function, the formula in \hyperref[FS height over the rational]{Corollary \ref*{FS height over the rational}} appears arduous to evaluate. It would anyway be interesting to use it to study arithmetical properties of projective hypersurfaces or recover similar results to the ones obtained in \cite{CM}.

\bibliographystyle{amsalpha}
\bibliography{./bibliography.bib}

\providecommand{\bysame}{\leavevmode\hbox to3em{\hrulefill}\thinspace}
\providecommand{\MR}{\relax\ifhmode\unskip\space\fi MR }
% \MRhref is called by the amsart/book/proc definition of \MR.
\providecommand{\MRhref}[2]{%
  \href{http://www.ams.org/mathscinet-getitem?mr=#1}{#2}
}
\providecommand{\href}[2]{#2}
\begin{thebibliography}{KKMS73}

\bibitem[AW89]{AW}
H.~Attouch and R.~J.-B. Wets, \emph{Epigraphical analysis}, Ann. Inst. H.
  Poincar\'e Anal. Non Lin\'eaire \textbf{6} (1989), no.~suppl., 73--100,
  Analyse non lin\'eaire (Perpignan, 1987).

\bibitem[Ber90]{Ber}
V.~G. Berkovich, \emph{Spectral theory and analytic geometry over
  non-{A}rchimedean fields}, Mathematical Surveys and Monographs, vol.~33,
  American Mathematical Society, Providence, RI, 1990.

\bibitem[Boy98]{Boyd}
D.~W. Boyd, \emph{Mahler's measure and special values of {$L$}-functions},
  Experiment. Math. \textbf{7} (1998), no.~1, 37--82.

\bibitem[BPRS15]{BPRS}
J.~I. {Burgos Gil}, P.~{Philippon}, J.~{Rivera-Letelier}, and M.~{Sombra},
  \emph{{The distribution of Galois orbits of points of small height in toric
  varieties}}, ArXiv e-prints (2015).

\bibitem[BPS14]{BPS}
J.~I. {Burgos Gil}, P.~Philippon, and M.~Sombra, \emph{Arithmetic geometry of
  toric varieties. {M}etrics, measures and heights}, Ast\'erisque (2014),
  no.~360, vi+222.

\bibitem[CD12]{CLD}
A.~{Chambert-Loir} and A.~{Ducros}, \emph{{Formes diff\'erentielles r\'eelles
  et courants sur les espaces de Berkovich}}, ArXiv e-prints (2012).

\bibitem[{Cha}06]{C-L}
A.~{Chambert-Loir}, \emph{Mesures et \'equidistribution sur les espaces de
  {B}erkovich}, J. Reine Angew. Math. \textbf{595} (2006), 215--235.

\bibitem[{Cha}11]{C-L1}
\bysame, \emph{Heights and measures on analytic spaces. {A} survey of recent
  results, and some remarks}, Motivic integration and its interactions with
  model theory and non-{A}rchimedean geometry. {V}olume {II}, London Math. Soc.
  Lecture Note Ser., vol. 384, Cambridge Univ. Press, Cambridge, 2011,
  pp.~1--50.

\bibitem[CLS11]{CLS}
D.~A. Cox, J.~B. Little, and H.~K. Schenck, \emph{Toric varieties}, Graduate
  Studies in Mathematics, vol. 124, American Mathematical Society, Providence,
  RI, 2011.

\bibitem[CM00]{CM}
J.~Cassaigne and V.~Maillot, \emph{Hauteur des hypersurfaces et fonctions
  z\^eta d'{I}gusa}, J. Number Theory \textbf{83} (2000), no.~2, 226--255.

\bibitem[CT09]{CLT}
A.~{Chambert-Loir} and A.~Thuillier, \emph{Mesures de {M}ahler et
  \'equidistribution logarithmique}, Ann. Inst. Fourier (Grenoble) \textbf{59}
  (2009), no.~3, 977--1014.

\bibitem[Den97]{Dening}
C.~Deninger, \emph{Deligne periods of mixed motives, {$K$}-theory and the
  entropy of certain {${\bf Z}^n$}-actions}, J. Amer. Math. Soc. \textbf{10}
  (1997), no.~2, 259--281.

\bibitem[EKL06]{EKL}
M.~Einsiedler, M.~Kapranov, and D.~Lind, \emph{Non-{A}rchimedean amoebas and
  tropical varieties}, J. Reine Angew. Math. \textbf{601} (2006), 139--157.

\bibitem[Ewa96]{Ewald}
G.~Ewald, \emph{Combinatorial convexity and algebraic geometry}, Graduate Texts
  in Mathematics, vol. 168, Springer-Verlag, New York, 1996.

\bibitem[Fal91]{Fal2}
G.~Faltings, \emph{Diophantine approximation on abelian varieties}, Ann. of
  Math. (2) \textbf{133} (1991), no.~3, 549--576.

\bibitem[Ful93]{Ful}
W.~Fulton, \emph{Introduction to toric varieties}, Annals of Mathematics
  Studies, vol. 131, Princeton University Press, Princeton, NJ, 1993, The
  William H. Roever Lectures in Geometry.

\bibitem[Ful98]{Ful2}
\bysame, \emph{Intersection theory}, second ed., Ergebnisse der Mathematik und
  ihrer Grenzgebiete. 3. Folge. A Series of Modern Surveys in Mathematics
  [Results in Mathematics and Related Areas. 3rd Series. A Series of Modern
  Surveys in Mathematics], vol.~2, Springer-Verlag, Berlin, 1998.

\bibitem[GH15]{GH}
W.~{Gubler} and J.~{Hertel}, \emph{{Local heights of toric varieties over
  non-archimedean fields}}, ArXiv e-prints (2015).

\bibitem[GK17]{GK}
W.~Gubler and K.~K\"unnemann, \emph{A tropical approach to nonarchimedean
  {A}rakelov geometry}, Algebra Number Theory \textbf{11} (2017), no.~1,
  77--180.

\bibitem[GKZ08]{GKZ}
I.~M. Gelfand, M.~M. Kapranov, and A.~V. Zelevinsky, \emph{Discriminants,
  resultants and multidimensional determinants}, Modern Birkh\"auser Classics,
  Birkh\"auser Boston, Inc., Boston, MA, 2008, Reprint of the 1994 edition.

\bibitem[Gub97]{G1}
W.~Gubler, \emph{Heights of subvarieties over {$M$}-fields}, Arithmetic
  geometry ({C}ortona, 1994), Sympos. Math., XXXVII, Cambridge Univ. Press,
  Cambridge, 1997, pp.~190--227.

\bibitem[Gub98]{G0}
\bysame, \emph{Local heights of subvarieties over non-{A}rchimedean fields}, J.
  Reine Angew. Math. \textbf{498} (1998), 61--113.

\bibitem[Gub03]{G}
\bysame, \emph{Local and canonical heights of subvarieties}, Ann. Sc. Norm.
  Super. Pisa Cl. Sci. (5) \textbf{2} (2003), no.~4, 711--760.

\bibitem[Gub07]{G2}
\bysame, \emph{Tropical varieties for non-{A}rchimedean analytic spaces},
  Invent. Math. \textbf{169} (2007), no.~2, 321--376.

\bibitem[KKMS73]{KKMS}
G.~Kempf, F.~F. Knudsen, D.~Mumford, and B.~{Saint-Donat}, \emph{Toroidal
  embeddings. {I}}, Lecture Notes in Mathematics, Vol. 339, Springer-Verlag,
  Berlin-New York, 1973.

\bibitem[Lal08]{Lalin}
M.~N. Lal\'{\i}n, \emph{Mahler measures and computations with regulators}, J.
  Number Theory \textbf{128} (2008), no.~5, 1231--1271.

\bibitem[Lek69]{Lekk}
C.~G. Lekkerkerker, \emph{Geometry of numbers}, Bibliotheca Mathematica, Vol.
  VIII, Wolters-Noordhoff Publishing, Groningen; North-Holland Publishing Co.,
  Amsterdam-London, 1969.

\bibitem[Mai00]{Mail}
V.~Maillot, \emph{G\'eom\'etrie d'{A}rakelov des vari\'et\'es toriques et
  fibr\'es en droites int\'egrables}, M\'em. Soc. Math. Fr. (N.S.) (2000),
  no.~80, vi+129.

\bibitem[MS16]{MS}
C.~{Mart{\'{\i}}nez} and M.~{Sombra}, \emph{{An arithmetic
  Bern\v{s}tein-Ku\v{s}nirenko inequality}}, ArXiv e-prints (2016).

\bibitem[Oda88]{O}
T.~Oda, \emph{Convex bodies and algebraic geometry---toric varieties and
  applications. {I}}, Algebraic {G}eometry {S}eminar ({S}ingapore, 1987), World
  Sci. Publishing, Singapore, 1988, pp.~89--94.

\bibitem[Phi95]{PhilIII}
P.~Philippon, \emph{Sur des hauteurs alternatives. {III}}, J. Math. Pures Appl.
  (9) \textbf{74} (1995), no.~4, 345--365.

\bibitem[PR04]{PR}
M.~Passare and H.~Rullg{\aa}rd, \emph{Amoebas, {M}onge-{A}mp\`ere measures, and
  triangulations of the {N}ewton polytope}, Duke Math. J. \textbf{121} (2004),
  no.~3, 481--507.

\bibitem[PS08a]{PS1}
P.~Philippon and M.~Sombra, \emph{Hauteur normalis\'ee des vari\'et\'es
  toriques projectives}, J. Inst. Math. Jussieu \textbf{7} (2008), no.~2,
  327--373.

\bibitem[PS08b]{PS2}
\bysame, \emph{A refinement of the {B}ern\v stein-{K}u\v snirenko estimate},
  Adv. Math. \textbf{218} (2008), no.~5, 1370--1418.

\bibitem[Roc70]{R}
R.~T. Rockafellar, \emph{Convex analysis}, Princeton Mathematical Series, No.
  28, Princeton University Press, Princeton, N.J., 1970.

\bibitem[RT77]{RT}
J.~Rauch and B.~A. Taylor, \emph{The {D}irichlet problem for the
  multidimensional {M}onge-{A}mp\`ere equation}, Rocky Mountain J. Math.
  \textbf{7} (1977), no.~2, 345--364.

\bibitem[Smy81]{Smyth}
C.~J. Smyth, \emph{On measures of polynomials in several variables}, Bull.
  Austral. Math. Soc. \textbf{23} (1981), no.~1, 49--63.

\bibitem[Zha95]{Z}
S.~Zhang, \emph{Small points and adelic metrics}, J. Algebraic Geom. \textbf{4}
  (1995), no.~2, 281--300.

\end{thebibliography}

\end{document}